\documentclass[english]{ourlematemasingle}

\usepackage{hyperref}

\usepackage[T1]{fontenc}
\usepackage[utf8]{inputenc}
\usepackage{anyfontsize}

\usepackage{amsmath,amssymb,amsthm}
\usepackage[usenames,dvipsnames]{xcolor}
\usepackage{paralist}

\newif\ifkeepslowthings
\keepslowthingstrue

\definecolor{cb-yellow}{RGB}{221,170,51}
\definecolor{cb-red} {RGB}{187,85,102}
\definecolor{cb-teal}{RGB}{0,153,136}
\definecolor{cb-blue} {RGB}{0,68,136}
\definecolor{cb-green}{RGB}{17,119,51}
\definecolor{cb-purple} {RGB}{170,68,153}
\definecolor{cb-palegrey} {RGB}{221,221,221}

\usepackage{tikz,pgfplots}
\usetikzlibrary{positioning}
\pgfplotsset{compat=1.18}

\hypersetup{
  pdfauthor={Michael Borinsky, Chiara Meroni, Maximilian Wiesmann},
  pdftitle={Bivariate exponential integrals and edge-bicolored graphs},
  pdfsubject={},
  pdfkeywords={},
    colorlinks=true,
    linkcolor=cb-red,
    filecolor=magenta,      
    urlcolor=cb-yellow,
    citecolor=cb-yellow,
}

\newtheorem{theorem}{Theorem}
\numberwithin{theorem}{section}

\newtheorem{proposition}[theorem]{Proposition}
\newtheorem{lemma}[theorem]{Lemma}
\newtheorem{conjecture}[theorem]{Conjecture}
\newtheorem{definition}[theorem]{Definition}

\theoremstyle{remark}
\newtheorem{remark}[theorem]{Remark}
\newtheorem{examplex}[theorem]{Example}
\newenvironment{example}
{\pushQED{\qed}\examplex}
{\popQED\endexamplex}

\newcommand{\colS}{red}
\newcommand{\colT}{yellow}

\newcommand{\GG}{\mathcal{G}}
\newcommand{\GGA}{\mathcal{G}^\star}
\newcommand{\PP}{\mathcal{P}}

\DeclareMathOperator{\SG}{{\mathbb S}}

\newcommand{\R}{\mathbb{R}}
\newcommand{\Q}{\mathbb{Q}}

\newcommand{\C}{\mathbb{C}}
\newcommand{\Z}{\mathbb{Z}}
\newcommand{\dd}{\mathrm{d}}

\DeclareMathOperator{\crit}{crit}
\DeclareMathOperator{\Aut}{Aut}

\newcommand{\ol}{\Lambda}

\newcommand{\bigO}{\mathcal{O}}

\renewcommand{\d}{\mathrm{d}}

\newcommand{\G}{\mathcal{G}}

\newcommand\restr[2]{{
  \left.\kern-\nulldelimiterspace 
  #1
  \vphantom{\big|} 
  \right|_{#2}
  }}

\title{Bivariate exponential integrals and edge-bicolored graphs}

  \author{Michael Borinsky}
  \address{%
  ETH Institute for Theoretical Studies \& \\
  Perimeter Institute \\
\email{michael.borinsky@eth-its.ethz.ch}
}

  \author{Chiara Meroni}
  \address{%
  ETH Institute for Theoretical Studies \\
\email{chiara.meroni@eth-its.ethz.ch}
}
\authorline 

\author{Maximilian Wiesmann}
\address{%
    Max Planck Institute for Mathematics in the Sciences \\
\email{wiesmann@mis.mpg.de}
}

\begin{document}

\begin{abstract}
We show that specific exponential bivariate integrals serve as generating functions of labeled edge-bicolored graphs.
Based on this, we prove an asymptotic formula for the number of regular edge-bicolored graphs with arbitrary weights assigned to different vertex structures. The asymptotic behavior is governed by the critical points of a polynomial.
As an application, we discuss the Ising model on a random 4-regular graph and show how its phase transitions arise from our formula.
\end{abstract}

\section{Introduction}
In this article we study the following family of bivariate integrals,
\begin{equation}
\label{equ:exponentialIntegral}
I(z) = \frac{z}{2\pi}
\int_D \exp\left( {z\, g(x,y)} \right) \, \dd x \dd y\,,
\end{equation}
where $D$ is some subset of $\R^2$, $g$ is a function $D\rightarrow \R$ fulfilling specific conditions and $z$ is a large positive number. Integrals as $I(z)$ arise naturally in two important applications. First, they appear in Bayesian statistics as \emph{marginal likelihood integrals} (see, e.g.,\ \cite[\S 1]{watanabe2009algebraic}). Second, they are \emph{path integrals} associated to a zero-dimensional quantum system with two interacting fields parametrized by $x$ and $y$, whose action is given by $g(x,y)$ (see, e.g.,~\cite[\S 2]{skinner2018quantum} or \cite{bessis1980quantum}). The setups might differ, however, in the integration domain $D$, leading to different asymptotic behaviors (see \cite{lin2011algebraic} for an asymptotic analysis in the realm of statistics).
Edge-bicolored graphs are central objects in extremal graph theory (see, e.g., \cite[Ch.~V]{bollobas2004extremal}) and their (asymptotic) enumeration is a subject with a long history  (see, e.g., \cite{MR3966531} and the references therein).\par 

We will explain that the coefficients of the large-$z$ asymptotic expansion of $I(z)$ count weighted edge-bicolored graphs.  Each graph is weighted by the reciprocal of the order of its automorphism group and the product of an arbitrary set of parameters assigned to each bicolored incidence structure of a vertex.  We do so by proving a bivariate version of the Laplace method in Section~\ref{sec:laplace}, before interpreting the coefficients of the asymptotic expansion combinatorially in Section~\ref{sec:graphs}.  Therefore, we may interpret $I(z)$ as a \emph{generating function} of edge-bicolored graphs (Theorem \ref{thm:laplace-graphs}).
In Section~\ref{sec:algo}, we derive an effective algorithm for the computation of these coefficients. In the final Section~\ref{sec:asymptotics}, we prove an asymptotic formula for the weighted number of \emph{regular} edge-bicolored graphs, in the limit where the number of edges and vertices goes to infinity.
Our main result Theorem \ref{thm:crit_g} relates this asymptotic formula to the critical points of the polynomial $g(x,y)$ whose shape is governed by the vertex incidence structure of the graphs. We showcase that, unlike the monochromatic case, which has previously been discussed in \cite[Ch.~3]{Borinsky:2018mdl}, only critical points satisfying some reality constraints contribute to the asymptotics.\par 

Throughout the text we illustrate our statements through the example of the Ising model on a random 4-regular graph. The Ising model is a central object of study in mathematical physics (see, e.g., \cite{MR4680248}). The relationship between our combinatorial approach and this model is explained in Remark~\ref{rmk:ising}.\par\medskip

\textbf{Acknowledgements.} 
MW thanks the ETH Zürich Institute for Theoretical Studies for generous support during two research visits.
MB and CM are supported by Dr.\ Max Rössler, the Walter Haefner Foundation, and the ETH Zürich Foundation. 
Research at Perimeter Institute is supported in part by the Government of Canada
through the Department of Innovation, Science and Economic Development and by
the Province of Ontario through the Ministry of Colleges and Universities.

\section{Laplace method and asymptotic expansions}
\label{sec:laplace}
We will start by using the Laplace method to study the large $z$ behavior of 
the integral $I(z)$ defined in \eqref{equ:exponentialIntegral}.
We require the data $D$, and $g: D\rightarrow \R$ determining $I(z)$ to be chosen such that
\begin{enumerate}
\item the integral $I(z)$ exists for $z > 0$, 
\item 
$D$ contains a neighborhood of the origin,
\item 
$g$ attains its unique supremum $\sup_{(x,y)\in D}g(x,y) \!=\! g(\mathbf{0})$ at the origin,
\item 
near the origin, $g$ has a converging  expansion of the form
\begin{equation}
\label{eq:gdef}
g(x,y) = - \frac{x^2}{2} - \frac{y^2}{2} + \sum_{\substack{u,w \geq 0\\ u+w \geq 3}} \ol_{u,w} \frac{x^u y^w}{u!w!}\, .
\end{equation}
\end{enumerate}
The last condition ensures that \eqref{equ:exponentialIntegral} resembles a Gaussian integral 
when $x$ and $y$ are small. This observation allows to approximate $I(z)$ by a slightly perturbed Gaussian
when $z$ is large.\par 

We define a family of polynomials $a_{s,t}$ indexed by integers $s,t \geq 0$ in a two-fold infinite set of variables $\lambda_{u,w}$ indexed by $u,w \geq 0$ with $u+w \geq 1$.  Let $\mathcal R$ be the ring of polynomials in these variables, i.e., $\mathcal R = \Q[\lambda_{0,1},\lambda_{1,0},\lambda_{1,1}, \lambda_{0,2}, \ldots]$.
The polynomials $a_{s,t}(\lambda) \in \mathcal R$ are defined by the generating function
\begin{equation}
\label{eq:coeffs_extraction_gen_funct}
    \sum_{s,t \geq 0} a_{s,t}(\lambda)\, x^s y^t =
\exp\left(\sum_{\substack{u,w \geq 0\\u+w \geq 1}} \lambda_{u,w} \frac{x^u y^w}{u!w!}\right) \in \mathcal R[[x,y]].
\end{equation}
For instance,
$a_{0,0}(\lambda) = 1$,
$a_{1,0}(\lambda) = \lambda_{1,0}$,
and 
$a_{2,0}(\lambda) = \frac12 ( \lambda_{2,0} + \lambda_{1,0}^2 )$.

We will relate the \emph{asymptotic expansion} of $I(z)$ for large $z$ to the polynomials $a_{s,t}$.
For a given function $h(z)$, the set $\bigO(h(z))$ consists of all functions
$f(z)$ for which $\limsup_{z\rightarrow \infty} \left|{f(z)}/{h(z)}\right|$ is finite. 
The notation $f(z) = g(z) + \bigO(h(z))$ means that $f(z)-g(z) \in \bigO(h(z))$. 
The asymptotic expansion notation
$f(z) \sim \sum_{n \geq 0} g_n(z)$  means that  $f(z) - \sum_{n=0}^{R-1} g_n(z) \in \bigO(g_{R}(z))$
for all $R \geq 0$.

\begin{proposition}
\label{prop:laplace}
If $I(z)$, $g$, $D$ and the coefficients $\ol_{u,w}$ are related as above, then
\[
I(z) \sim \sum_{n\geq 0} A_n z^{-n},
\]
for large $z$, where $A_n$ is the coefficient of $z^{-n}$ in the formal power series
\[
    \sum_{s,t\geq 0} z^{-(s+t)} (2s-1)!! \cdot (2t-1)!! \cdot a_{2s,2t} (z \cdot \ol) \in \R[[z^{-1}]]\,,
\]
where $(2s-1)!! = (2s-1)(2s-3)\cdots 3\cdot 1$ and $a_{2s,2t} (z \cdot \ol)$ is the polynomial $a_{2s,2t}(\lambda)$ defined in \eqref{eq:coeffs_extraction_gen_funct},
with 
\begin{equation}
    \label{eq:substitution_lambda_Lambda}
    \lambda_{u,w} = \left\{ \begin{array}{ll}
        0 & \text{for}~ u,w\geq 0 ~\text{and}~ 1 \leq u+w < 3, \\
        z\,\Lambda_{u,w} & \text{for}~ u,w\geq 0 ~\text{and}~ u+w \geq 3 .\\
    \end{array} \right.
\end{equation}
\end{proposition}

The proof of this proposition uses the classical \emph{Laplace method} which gives an expression for the asymptotic expansion of the integral $I(z)$.  See, e.g., \cite[Appendix A]{borinsky2019euler} for the proof of the one-dimensional case.

\begin{proof}
Fix an integer $R \geq 0$ and any value for $\gamma \in \left(\frac13,\frac12\right)$.
We first prove that the integral $I(z)$ is concentrated in the square $B(z) = [-z^{-\gamma},z^{-\gamma}]^2 \subset D$ that shrinks for growing $z$.
Let $\delta(z) = \max_{(x,y) \in D \setminus B(z)} g(x,y)$, then
    \begin{gather*}
    \left\vert I(z) - \frac{z}{2\pi} \int_{B(z)} \exp(z g(x,y)) \d x \d y \right\vert 
    = \frac{z}{2\pi} \int_{D\setminus B(z)} \exp(z g(x,y)) \d x \d y \\
    \leq \frac{z}{2\pi} \exp\left((z-1) \delta(z)\right) \int_{D} \exp(g(x,y)) \d x \d y \,.
    \end{gather*}
The last integral is finite by requirement.
As the origin is the unique global maximum of $g$ in $D$,
the maximal value $\delta(z)$ will be attained on the boundary of the square
$B(z)$ if $z$ is sufficiently large. Near the origin $g(x,y)$ behaves as \hbox{$-\frac{x^2}{2}-\frac{y^2}{2} +\textrm{(higher order terms)}$}, so $\delta(z)= -\frac{1}{2}z^{-2\gamma} + \bigO(z^{-3\gamma})$. Hence,  
    \begin{gather}
    \label{eq:Iz1}
    I(z) =  \frac{z}{2\pi} \int_{B(z)} \exp(z g(x,y)) \d x \d y  + \bigO(z \exp( -z^{1-2\gamma})).
    \end{gather}
As $\gamma < \frac12$, we have, in particular, $\bigO(z \exp( -z^{1-2\gamma})) \subset \bigO(z^{-R})$.
So, for the purpose of finding the asymptotic expansion of $I(z)$ in decreasing powers $z^{0},$ $z^{-1}, z^{-2},\ldots,z^{-R+1}$, integrating only over $B(z)$ as in \eqref{eq:Iz1} is sufficient.

Note that $a_{s,t}(z \cdot \ol)$ is a polynomial of degree at most $\frac{s+t}{3}$ in $z$. The function $\exp(z (\frac{1}{2}x^2+\frac{1}{2}y^2+g(x,y)))$ is analytic for all $(x,y) \in B(z)$. Therefore, for each $K \geq 0$, there is a constant $C> 0$ such that 
    \begin{align*}
        & \left\lvert  \exp\left(z \sum_{\substack{u,w \geq 0\\ u+w \geq 3}} \ol_{u,w} \frac{x^u y^w}{u!w!} \right) - \sum_{\substack{s,t \geq 0\\ s+t < K}} a_{s,t}(z \cdot \ol) x^{s} y^t \right\rvert
         \leq C z^{\frac{1}{3}K-\gamma K} \text{ for }  (x,y) \in B(z).
    \end{align*}
Next, we fix $K=\frac{3R}{3\gamma -1} \geq 0$ so that 
$z^{\frac{1}{3}K-\gamma K} = z^{-R}$, and use \eqref{eq:Iz1} to get
\begin{align}
\label{eq:Iz2}
    I(z) =  \frac{z}{2\pi} \sum_{\substack{s,t \geq 0\\ s+t < K}}  a_{s,t}(z \cdot \ol) 
\int_{B(z)} e^{-z\frac{x^2}{2} - z\frac{y^2}{2} }  
x^{s} y^t
\d x \d y  + \bigO(z^{-R}).
\end{align}
We want now to extend the integration domain to the whole real plane. For any integer $s \geq 0$, consider the integral
$$
\int_{z^{-\gamma}}^\infty e^{-z\frac{x^2}{2}}x^s \d x =
\exp\left({-\frac{z^{1-2\gamma}}{2}}\right)\int_{0}^\infty \exp\left({-z\frac{x^2}{2} -z^{1-\gamma} x}\right)(z^{-\gamma}+x)^s \d x.
$$
For fixed $z$, the function $x\mapsto \exp({-z^{1-\gamma} x})(z^{-\gamma}+x)^s$ attains its unique maximum at
$x=x_\mathrm{max} = sz^{\gamma-1} -z^{-\gamma}$. If $z$ is sufficiently large we have $x_\mathrm{max} \leq 0$. Hence, in the range we are interested in, the integral is decreasing in $x$, and using $\sqrt{\frac{z}{2\pi}}\int_\R e^{-z \frac{x^2}{2}} \d x = 1$ we get
$$
\sqrt{\frac{z}{2\pi}}
\int_{z^{-\gamma}}^\infty e^{-z\frac{x^2}{2}}x^s \d x 
\in \bigO\left( 
z^{-\gamma s}\,
\exp\left({-\frac{z^{1-2\gamma}}{2}}\right)
\right) \subset \bigO(z^{-R}).
$$
Combining this with \eqref{eq:Iz2} shows that 
\begin{align*}
    I(z) =  \frac{z}{2\pi} \sum_{\substack{s,t \geq 0\\ s+t < K}}  a_{s,t}(z \cdot \ol) 
\int_{\R^{2}} e^{-z\frac{x^2}{2} - z\frac{y^2}{2} }  
x^{s} y^t
\d x \d y  + \bigO(z^{-R}).
\end{align*}
Using the Gaussian integral identities $\sqrt{\frac{z}{2\pi}}\int_\R e^{-z \frac{x^2}{2}}x^{2s} \d x = z^{-s}\cdot (2s-1)!!$
and $\int_\R e^{-z \frac{x^2}{2}}x^{2s+1} \d x=0$ for all integers $s\geq 0$, proves the statement.
\end{proof}

\begin{example}
\label{ex:Ising_model_deg4_part1}
Fix $D = [-1,1]^2$ and $g(x,y) = -\frac{x^2}{2} - \frac{y^2}{2} + \frac{x^4}{4!} + \lambda \frac{x^2 y^2}{4} + \lambda^2 \frac{y^4}{4!}$ with $\lambda$ some arbitrary constant in $\R$. 
The conditions for Proposition~\ref{prop:laplace} are fulfilled and the associated integral $I(z)$ as defined in \eqref{equ:exponentialIntegral} has an asymptotic expansion
$ I(z) \sim \sum_{n \geq 0} A_n z^{-n}$.
Using the formula from Proposition~\ref{prop:laplace} and the generating function for the polynomials $a_{s,t}$ from \eqref{eq:coeffs_extraction_gen_funct}, we find that $A_0=1$ and
    \begin{align*}
        A_1 &= \frac{1}{8} + \frac{1}{4}\lambda + \frac{1}{8}\lambda^2, \\
        A_2 &= \frac{35}{384} + \frac{5}{32}\lambda + \frac{19}{64}\lambda^2 + \frac{5}{32}\lambda^3 + \frac{35}{384}\lambda^4,\\
        A_3 &= \frac{385}{3072} + \frac{105}{512}\lambda + \frac{1295}{3072}\lambda^2 + \frac{175}{256}\lambda^3 + \frac{1295}{3072}\lambda^4 + \frac{105}{512}\lambda^5 + \frac{385}{3072}\lambda^6.
    \qedhere
    \end{align*}
\end{example}

In the next section, we will endow the obtained analytic expressions with a combinatorial interpretation. This process is inspired by quantum field theory in physics, where perturbative expansions of observables, which are combinatorially controlled via \emph{Feynman graphs}, relate to \emph{path integrals}. The integral in \eqref{equ:exponentialIntegral} can be seen as a zero-dimensional path integral. The associated Feynman graphs are \emph{edge-bicolored graphs}.

\section{Edge-bicolored graphs}
\label{sec:graphs}

A \emph{graph} is a one-dimensional, finite CW complex. It is \emph{edge-bicolored} if each edge has one of two different colors. We will represent graphs using only discrete data.
A \emph{(set) partition} $P$ of a finite set $H$ is a set of non-empty and mutually disjoint subsets of $H$
whose union equals $H$. The elements of $P$ are called \emph{blocks}.
\begin{definition}
\label{def:bigraph}
Given two disjoint finite sets $S$ and $T$ of labels, an \emph{$[S,T]$-labeled edge-bicolored graph} is a tuple $\Gamma = (V,E_S,E_T)$, where
\begin{enumerate}
\item the \emph{vertex set} $V$ is a partition of $S\sqcup T$,
\item $E_S$ is a partition of $S$ into blocks of size 2,
\item $E_T$ is a partition of $T$ into blocks of size 2.
\end{enumerate}
\end{definition}
We think of the elements of $S$ and $T$ as \emph{half-edge labels} colored 
\colS{} and \colT{}, respectively.
These half-edges are bundled together in vertices via the partition $V$. 
The edge sets $E_S$ and $E_T$ pair the half-edges into edges of the respective color. Every edge-bicolored graph without isolated vertices can be represented by at least one $[S,T]$-labeled graph.
All graphs in this article will be edge-bicolored, so we will drop this adjective from now on. 

    \begin{figure}
        \centering
        \scalebox{0.5}{\begin{tikzpicture}[x=1ex,y=1ex,baseline={([yshift=-1.5ex]current bounding box.center)}]
    \coordinate (v);

    \coordinate [left=8 of v] (vm1);
    \coordinate [left=3 of vm1] (v01);
    \coordinate [right=3 of vm1] (v11);
    \draw[cb-red, line width = 3.5] (v01) circle(3);
    \draw[cb-red, line width = 3.5] (v11) circle(3);
    \filldraw (vm1) circle (3pt);

    \coordinate [right=8 of v] (vm2);
    \coordinate [left=3 of vm2] (v02);
    \coordinate [right=3 of vm2] (v12);
    \draw[cb-red, line width = 3.5] (v02) circle(3);
    \draw[cb-yellow, line width = 3.5] (v12) circle(3);
    \filldraw (vm2) circle (3pt);
\end{tikzpicture}%
}
        \caption{An edge-bicolored graph with two connected components.}
    \label{fig:extwotadpole2s_rrry}
    \end{figure}
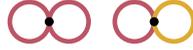

\begin{example}
    Let $S = \{s_1,s_2,\dots,s_6\}$ and $T = \{t_1,t_2\}$. The partitions 
    \begin{align*}
        V & = \left\{ \{s_1,s_2,s_3,s_4\}, \{s_5,s_6,t_1,t_2\} \right\}, \\
        E_S & = \left\{ \{s_1,s_2\}, \{s_3,s_4\}, \{s_5,s_6\} \right\}, \\
        E_T & = \left\{ \{t_1,t_2\} \right\},
    \end{align*}
    form an $[S,T]$-labeled graph representing the graph depicted in Figure \ref{fig:extwotadpole2s_rrry}.
\end{example}

An \emph{isomorphism} from an $[S_1,T_1]$-labeled graph $(V^{1},E_S^{1},E_T^{1})$ to an $[S_2,T_2]$-labeled graph $(V^{2},E_S^{2},E_T^{2})$
is a pair of bijections $j_S:S_1\rightarrow S_2$, $j_T:T_1\rightarrow T_2$ such that $j(V^{1}) = V^{2}, j(E_S^{1}) = E_S^{2}$, and $j(E_T^{1}) = E_T^{2}$
with $j$ being the map that $j_S$ and $j_T$ canonically induce on the subsets of $S$, $T$, and $S\sqcup T$.
An \emph{automorphism} of an $[S,T]$-lab\-eled graph $\Gamma$ is an isomorphism to itself. Those form the group $\Aut(\Gamma)$.
\begin{lemma}
\label{lmm:aut}
Each $[\{1,\ldots,2s\},\{1,\ldots,2t\}]$-labeled graph $\Gamma$ belongs to an isomorphism class of such graphs of size $\frac{(2s)!(2t)!}{|\Aut(\Gamma)|}$.
\end{lemma}
\begin{proof}
For given $\Gamma$, let $\mathrm{lab}(\Gamma)$ be the set of $[\{1,\ldots,2s\},\{1,\ldots,2t\}]$-labeled graphs that are isomorphic to $\Gamma$. The group $\SG_{2s}\times \SG_{2t}$ acts on $\mathrm{lab}(\Gamma)$ by permuting the half-edge labels of the respective color. $\Aut(\Gamma)$ is the subgroup of $\SG_{2s}\times \SG_{2t}$ stabilizing $\Gamma$. The lemma follows from the orbit stabilizer theorem.
\end{proof}

\begin{example}
    An $[S,T]$-labeled graph $\Gamma$ representing the graph depicted in Figure \ref{fig:extwotadpole2s_rrry} has automorphism group isomorphic to $\left(
    \SG_2 \times \SG_2 \rtimes \SG_2 \times \SG_2 \right) \times \SG_2 \leq \SG_6 \times \SG_2$, where $\rtimes$ denotes the semidirect product of groups, $\SG_6$ refers to the six {\colS} half-edges in $S$ and $\SG_2$ to the two {\colT} half-edges in $T$.
\end{example}

We write $\GG$ for the set of isomorphism classes of graphs. For each $G \in \GG$, we write $V^G,E_S^G,E_T^G,E^G=E_S^G\sqcup E_T^G$ and $\Aut(G)$ 
for the respective set or group of some $[S,T]$-labeled representative of $G$.
The \emph{Euler characteristic} of $G$ is defined by $\chi(G) = |V^G|-|E^G|$.
The \emph{bidegree} of a graph's vertex $v\in V^G$ is the pair of integers $\deg(v) = (u,w)$ where $u$ counts the number of half-edges in $v$ that lie in the \colS{}-colored set $S$ and $w$ the half-edges in the \colT{}-colored part $T$. The \emph{vertex degree} of $v$ is $|\deg(v)| = u+w$.

\begin{proposition}
\label{prop:genfun2colgraphs}
The generating function for graphs with marked bidegrees is
\begin{gather*}
        \sum_{G\in \G} \frac{\eta^{|E^G|} }{
\left\lvert \Aut(G) \right\rvert} 
\prod \limits _{v\in V^G}
\lambda_{\deg(v)}
= \sum_{s,t\geq 0} \eta^{s+t} \cdot  (2s-1)!!\cdot (2t-1)!!\cdot
a_{2s,2t}(\lambda) \in \mathcal R[[\eta]],
\end{gather*}
where $a_{s,t}$ is defined as in \eqref{eq:coeffs_extraction_gen_funct}.
\end{proposition}

We postpone the proof and first illustrate the result with an example.
\begin{example}
    The formula in Proposition \ref{prop:genfun2colgraphs} provides a recipe to count our graphs for a given number of edges, grouping them according to their bidegrees. For instance, the coefficient of $\eta^1$ counts graphs with one edge:
    \begin{align*}
    \sum_{\substack{G\in \G, \\ |E^G|=1}} \frac{1}{\left\lvert \Aut(G) \right\rvert} \prod \limits _{v\in V^G} \lambda_{\deg(v)} &= \;\;\; \scalebox{.3}{\begin{tikzpicture}[x=1ex,y=1ex,baseline={([yshift=-1.5ex]current bounding box.center)}]
    \coordinate (v);
    \coordinate [above=4 of v] (vm1);
    \coordinate [right=3 of vm1] (v11);
    \draw[cb-red, line width = 3.5] (v11) circle(3);
    \filldraw (vm1) circle (3pt);
\end{tikzpicture}%
} \;\;+\;\;\: \scalebox{.3}{\begin{tikzpicture}[x=1ex,y=1ex,baseline={([yshift=-1.5ex]current bounding box.center)}]
    \coordinate (vm);
    \coordinate [left=4 of vm] (v0);
    \coordinate [right=4 of vm] (v1);
    \draw[cb-red, line width = 3.5] (v0) to (v1);
    \filldraw (v0) circle (3pt);
    \filldraw (v1) circle (3pt);
\end{tikzpicture}%
} \:\;\;+\;\;\: \scalebox{.3}{\begin{tikzpicture}[x=1ex,y=1ex,baseline={([yshift=-1.5ex]current bounding box.center)}]
    \coordinate (v);
    \coordinate [above=4 of v] (vm1);
    \coordinate [right=3 of vm1] (v11);
    \draw[cb-yellow, line width = 3.5] (v11) circle(3);
    \filldraw (vm1) circle (3pt);
\end{tikzpicture}%
} \:\;\;+\;\; \scalebox{.3}{\begin{tikzpicture}[x=1ex,y=1ex,baseline={([yshift=-1.5ex]current bounding box.center)}]
    \coordinate (vm);
    \coordinate [left=4 of vm] (v0);
    \coordinate [right=4 of vm] (v1);
    \draw[cb-yellow, line width = 3.5] (v0) to (v1);
    \filldraw (v0) circle (3pt);
    \filldraw (v1) circle (3pt);
\end{tikzpicture}%
} \\
    &= \frac{1}{2} \lambda_{2,0} + \frac{1}{2} \lambda_{1,0}^2 + \frac{1}{2} \lambda_{0,2} + \frac{1}{2} \lambda_{0,1}^2.
    \end{align*}
    Using the power series on the right-hand side of Proposition \ref{prop:genfun2colgraphs}, this can be obtained simply as $a_{2,0}+a_{0,2}$, and by expanding the exponential in \eqref{eq:coeffs_extraction_gen_funct}, we get exactly the above expression. If for $|E^G|=1$ these two approaches may seem equally complicated, already for graphs with two edges it is clear that the use of the generating function speeds up the computation. In fact, there are seven (monochromatic) graphs with two edges: \begin{tikzpicture}[x=1ex,y=1ex,baseline={([yshift=-.5ex]current bounding box.center)}]
    \coordinate (v);
    \coordinate[above left=1.3 of v] (v1);
    \coordinate[below left=1.3 of v] (v2);
    \coordinate[above right=1.3 of v] (v3);
    \coordinate[below right=1.3 of v] (v4);

    \draw (v1) to (v3);
    \draw (v2) to (v4);
    \filldraw (v1) circle(1pt);
    \filldraw (v2) circle(1pt);
    \filldraw (v3) circle(1pt);
    \filldraw (v4) circle(1pt);
\end{tikzpicture}%
\, , \begin{tikzpicture}[x=1ex,y=1ex,baseline={([yshift=-.5ex]current bounding box.center)}]
    \coordinate (v);

    \coordinate [above=.9 of v] (vm1);
    \coordinate [right=.7 of vm1] (v11);
    \draw (v11) circle(.7);
    \filldraw (vm1) circle (1pt);

    \coordinate [below=.9 of v] (vm2);
    \coordinate [right=.7 of vm2] (v12);
    \draw (v12) circle(.7);
    \filldraw (vm2) circle (1pt);
\end{tikzpicture}%
\, , \begin{tikzpicture}[x=1ex,y=1ex,baseline={([yshift=-.5ex]current bounding box.center)}]
    \coordinate (v);
    \coordinate[above left=1.3 of v] (v1);
    \coordinate[above right=1.3 of v] (v3);

    \draw (v1) to (v3);
    \filldraw (v1) circle(1pt);
    \filldraw (v3) circle(1pt);

    \coordinate [below left=1.3 of v] (vm2);
    \coordinate [right=.7 of vm2] (v12);
    \draw (v12) circle(.7);
    \filldraw (vm2) circle (1pt); 

\end{tikzpicture}%
\, , \begin{tikzpicture}[x=1ex,y=1ex,baseline={([yshift=-.5ex]current bounding box.center)}]
    \coordinate (vm);
    \coordinate [left=.7 of vm] (v0);
    \coordinate [right=.7 of vm] (v1);
    \draw (v0) circle(.7);
    \draw (v1) circle(.7);
    \filldraw (vm) circle (1pt);
\end{tikzpicture}%
\, , \begin{tikzpicture}[x=1ex,y=1ex,baseline={([yshift=-.5ex]current bounding box.center)}]
    \coordinate (vm);
    \coordinate [left=1 of vm] (v0);
    \coordinate [right=1 of vm] (v1);
    \draw (v0) to[bend left=45] (v1);
    \draw (v0) to[bend right=45] (v1);
    \filldraw (v0) circle (1pt);
    \filldraw (v1) circle (1pt);
\end{tikzpicture}%
\, , \begin{tikzpicture}[x=1ex,y=1ex,baseline={([yshift=-.5ex]current bounding box.center)}]
    \coordinate (v);
    \coordinate[left=1 of v] (v1);
    \coordinate[right=1 of v] (v3);

    \draw (v1) to (v3);
    \filldraw (v1) circle(1pt);
    \filldraw (v3) circle(1pt);

    \coordinate [right=.7 of v3] (v12);
    \draw (v12) circle(.7);

\end{tikzpicture}%
\, , \begin{tikzpicture}[x=1ex,y=1ex,baseline={([yshift=-.5ex]current bounding box.center)}]
    \coordinate (v);
    \coordinate[left=2 of v] (v1);
    \coordinate[right=2 of v] (v3);

    \draw (v1) to (v3);
    \filldraw (v1) circle(1pt);
    \filldraw (v3) circle(1pt);
    \filldraw (v) circle(1pt);

\end{tikzpicture}%
, which turn into $23$ edge-bicolored graphs. On the other hand, a simple expansion of the exponential function gives
    \begin{align*}
        \sum_{\substack{G\in \G, \\ |E^G|=2}} &\frac{1}{\left\lvert \Aut(G) \right\rvert} \prod \limits _{v\in V^G} \lambda_{\deg(v)} = 3 a_{4,0} + a_{2,2} + 3 a_{0,4} =\\
        =& \lambda_{0,1} \lambda_{1,0} \lambda_{1,1}+\tfrac{\lambda_{0,1}^4}{8}+\tfrac{3 \lambda_{0,1}^2 \lambda_{0,2}}{4}+\tfrac{\lambda_{0,1}^2 \lambda_{1,0}^2}{4}+\tfrac{\lambda_{0,1}^2 \lambda_{2,0}}{4}\\&+\tfrac{\lambda_{0,1} \lambda_{0,3}}{2}
        +\tfrac{\lambda_{0,1} \lambda_{2,1}}{2}
        +\tfrac{3 \lambda_{0,2}^2}{8}+\tfrac{\lambda_{0,2} \lambda_{1,0}^2}{4}+\tfrac{\lambda_{0,2} \lambda_{2,0}}{4}+\tfrac{\lambda_{0,4}}{8}+\tfrac{\lambda_{1,0}^4}{8}\\
        &+\tfrac{3 \lambda_{1,0}^2 \lambda_{2,0}}{4}+\tfrac{\lambda_{1,0} \lambda_{1,2}}{2}+\tfrac{\lambda_{1,0} \lambda_{3,0}}{2}+\tfrac{\lambda_{1,1}^2}{2}+\tfrac{3 \lambda_{2,0}^2}{8}+\tfrac{\lambda_{2,2}}{4}+\tfrac{\lambda_{4,0}}{8}.\qedhere
    \end{align*}
\end{example}

To prove Proposition \ref{prop:genfun2colgraphs}, we use the following lemma on the number of partitions of a set where elements come in two different colors.
Let $S = \{1,\ldots,s\}$ and $T = \{1,\ldots,t\}$ and $\PP_{s,t}$ the set of partitions of the disjoint union $S\sqcup T$.
For each block $B$ of a partition $P \in \PP_{s,t}$ we define the \emph{bidegree} $\deg(B)$ of the block to be the pair of integers $(u,w)$ 
where $u$ is the number of elements from $S$ and $w$ the number of elements from $T$ in $B$.
\begin{lemma}
\label{lmm:biparititons}
Given $s,t \geq 0$, consider a set of non-negative integers $n_{u,w}$ indexed by pairs $u,w$ with $0 \leq u \leq s$, $0 \leq w \leq t$, $u+w \geq 1$, such that 
\[
\sum_{u} u \cdot n_{u,w} = s, \quad \sum_{w} w \cdot  n_{u,w} = t.
\]
The number of partitions in $\PP_{s,t}$ with exactly $n_{u,w}$ blocks of bidegree $(u,w)$ is
\[
\frac{s!t!}{\prod_{u,w} n_{u,w}! (u!)^{n_{u,w}} (w!)^{n_{u,w}}}.
\]
\end{lemma}
\begin{proof}
The group $\SG_{s} \times \SG_{t}$ acts on $\PP_{s,t}$ by permuting the elements of $S$ and $T$, respectively. This action is transitive if we restrict to partitions with specific block bidegree set $\{n_{u,w}\}_{u,w}$. A specific partition with given block bidegrees is stabilized by the subgroup that permutes the elements inside each block and blocks of the same size. This subgroup is isomorphic to $(\SG_{u} \times \SG_{w})^{n_{u,w}} \rtimes \SG_{n_{u,w}}$. The claim follows from the orbit stabilizer theorem.
\end{proof}
\begin{proof}[Proof of Proposition~\ref{prop:genfun2colgraphs}]
From \eqref{eq:coeffs_extraction_gen_funct},
and $e^X = \sum_{n \geq 0} \frac{X^n}{n!}$, we get 
\begin{align}
\label{eq:expanded-exp}
s! \cdot t! \cdot a_{s,t}(\lambda) = 
\sum_{\{n_{u,w}\}}
\frac{s!t!}{\prod_{u,w} n_{u,w}! (u!)^{n_{u,w}} (w!)^{n_{u,w}}}
\prod_{\substack{u,w\geq 0\\ u+w \geq 1}} \lambda_{u,w}^{n_{u,w}},
\end{align}
where the sum is over all sets of integers $\{n_{u,w}\}$ that fulfill the conditions for Lemma~\ref{lmm:biparititons} with respect to $s$ and $t$.\par 

We can match the elements of the set $S=\{1,\ldots,2s\}$ among each other in $(2s-1)!!$ ways and the ones of $T=\{1,\ldots,2t\}$ analogously. 
So, by Definition~\ref{def:bigraph}, Lemma~\ref{lmm:biparititons}, and \eqref{eq:expanded-exp}, the number of $[S,T]$-labeled graphs with exactly $n_{u,w}$ vertices of bidegree $u,w$ is $(2s-1)!! \cdot (2t-1)!!$ times the coefficient of the monomial 
$\prod_{u,w} \lambda_{u,w}^{n_{u,w}}$ in 
the polynomial  $(2s)! \cdot (2t)! \cdot a_{2s,2t}(\lambda) \in \mathcal R$.
The statement follows then from Lemma~\ref{lmm:aut}.
\end{proof}

Our first main result follows by combining Propositions~\ref{prop:laplace} and \ref{prop:genfun2colgraphs}.

\begin{theorem} 
\label{thm:laplace-graphs}
If $I(z)$, $g$, $D$ and the coefficients $\ol_{u,w}$ are related as in Section~\ref{sec:laplace}, then the 
integral in \eqref{equ:exponentialIntegral} has the asymptotic expansion
\[
I(z) \sim \sum_{n\geq 0} A_n z^{-n},
\]
for large $z$, with the coefficients $A_n$ given by
\[
A_n = \sum_{G \in \GGA_{-n}} 
        \frac{1 }{
\left\lvert \Aut(G) \right\rvert} 
\prod \limits _{v\in V^G}
\ol_{\deg(v)},
\]
where we sum over the set $\GGA_{-n}$ of all edge-bicolored graphs with vertex degrees at least $3$ and Euler characteristic equal to $-n$.
\end{theorem}
\begin{proof}
Using the fact that $\chi(G)=|V^G| - |E^G|$ we rewrite
\begin{equation*}
    \sum_{n\geq 0}A_nz^{-n} = \sum_{G\in \GGA_{-n}} \frac{z^{\chi(G)}}{|\Aut(G)|} \prod_{v\in V^G} \Lambda_{\deg(v)} = \sum_{G \in \GGA_{-n}} \frac{z^{-|E^G|}}{|\Aut(G)|} \prod_{v\in V^G} z\cdot \Lambda_{\deg(v)}.
\end{equation*}
Applying Proposition \ref{prop:genfun2colgraphs} for $\lambda_{u,w}$ as defined in \eqref{eq:substitution_lambda_Lambda}, this is further equal to
\[
    \sum_{s,t\geq 0} z^{-(s+t)}\cdot (2s-1)!!\cdot (2t-1)!!\cdot a_{2s,2t}(\lambda).
\]
By Proposition \ref{prop:laplace}, this is the large-$z$ asymptotic expansion of $I(z)$.
\end{proof}

\begin{example}\label{ex:Ising_model_deg4_part2}
Continuing Example~\ref{ex:Ising_model_deg4_part1},
let $c^{(k)}_n$ be the coefficient of $\lambda^k$ in $A_n$. By Theorem~\ref{thm:laplace-graphs}, $c^{(k)}_n$ counts automorphism-weighted graphs with Euler characteristic $-n$ and vertex degree four, such that $k_1$ vertices have exactly two yellow half-edges and $k_2$ vertices have four yellow half-edges, so that $k_1+k_2 = k$. We can view this explicitly for $n=2$, as follows. Among the $21$ (monochromatic) graphs with $\chi = -2$, there are only three $4$-regular graphs. These are \begin{tikzpicture}[x=1ex,y=1ex,baseline={([yshift=-.5ex]current bounding box.center)}]
    \coordinate (v);

    \coordinate [above=1.2 of v] (vm1);
    \coordinate [left=.7 of vm1] (v01);
    \coordinate [right=.7 of vm1] (v11);
    \draw (v01) circle(.7);
    \draw (v11) circle(.7);
    \filldraw (vm1) circle (1pt);

    \coordinate [below=1.2 of v] (vm2);
    \coordinate [left=.7 of vm2] (v02);
    \coordinate [right=.7 of vm2] (v12);
    \draw (v02) circle(.7);
    \draw (v12) circle(.7);
    \filldraw (vm2) circle (1pt);
\end{tikzpicture}%
, \begin{tikzpicture}[x=1ex,y=1ex,baseline={([yshift=-.5ex]current bounding box.center)}]
    \coordinate (vm);
    \coordinate [left=1 of vm] (v0);
    \coordinate [right=1 of vm] (v1);
    \draw (v0) to[bend left=45] (v1);
    \draw (v0) to[bend right=45] (v1);
    \draw (vm) circle(1);
    \filldraw (v0) circle (1pt);
    \filldraw (v1) circle (1pt);
\end{tikzpicture}%
, \begin{tikzpicture}[x=1ex,y=1ex,baseline={([yshift=-.5ex]current bounding box.center)}]
    \coordinate (vm);
    \coordinate [left=.7 of vm] (v0);
    \coordinate [right=.7 of vm] (v1);
    \coordinate [left=.7 of v0] (vc1);
    \coordinate [right=.7 of v1] (vc2);
    \draw (vc1) circle(.7);
    \draw (vc2) circle(.7);
    \draw (vm) circle(.7);
    \filldraw (v0) circle (1pt);
    \filldraw (v1) circle (1pt);
\end{tikzpicture}%
. Considering all bicolorings, we get
    \begin{align*}
        c_2^{(0)} &= \scalebox{.3}{\begin{tikzpicture}[x=1ex,y=1ex,baseline={([yshift=-1.5ex]current bounding box.center)}]
    \coordinate (v);

    \coordinate [above=4 of v] (vm1);
    \coordinate [left=3 of vm1] (v01);
    \coordinate [right=3 of vm1] (v11);
    \draw[cb-red, line width = 3.5] (v01) circle(3);
    \draw[cb-red, line width = 3.5] (v11) circle(3);
    \filldraw (vm1) circle (3pt);

    \coordinate [below=4 of v] (vm2);
    \coordinate [left=3 of vm2] (v02);
    \coordinate [right=3 of vm2] (v12);
    \draw[cb-red, line width = 3.5] (v02) circle(3);
    \draw[cb-red, line width = 3.5] (v12) circle(3);
    \filldraw (vm2) circle (3pt);
\end{tikzpicture}%
} \; +\;  \scalebox{.3}{\begin{tikzpicture}[x=1ex,y=1ex,baseline={([yshift=-1.5ex]current bounding box.center)}]
    \coordinate (vm);
    \coordinate [left=4 of vm] (v0);
    \coordinate [right=4 of vm] (v1);
    \draw[cb-red, line width = 3.5] (v0) to[bend left=45] (v1);
    \draw[cb-red, line width = 3.5] (v0) to[bend right=45] (v1);
    \draw[cb-red, line width = 3.5] (vm) circle(4);
    \filldraw (v0) circle (3pt);
    \filldraw (v1) circle (3pt);
\end{tikzpicture}%
} \;+\; \scalebox{.3}{\begin{tikzpicture}[x=1ex,y=1ex,baseline={([yshift=-1.5ex]current bounding box.center)}]
    \coordinate (vm);
    \coordinate [left=4 of vm] (v0);
    \coordinate [right=4 of vm] (v1);
    \coordinate [left=4 of v0] (vc1);
    \coordinate [right=4 of v1] (vc2);
    \draw[cb-red, line width = 3.5] (vc1) circle(4);
    \draw[cb-red, line width = 3.5] (vm) circle(4);
    \draw[cb-red, line width = 3.5] (vc2) circle(4);
    \filldraw (v0) circle (3pt);
    \filldraw (v1) circle (3pt);
\end{tikzpicture}%
} = \tfrac{1}{128} + \tfrac{1}{48} + \tfrac{1}{16} = \tfrac{35}{384}, \\
        c_2^{(1)} &= \scalebox{.3}{\begin{tikzpicture}[x=1ex,y=1ex,baseline={([yshift=-1.5ex]current bounding box.center)}]
    \coordinate (v);

    \coordinate [above=4 of v] (vm1);
    \coordinate [left=3 of vm1] (v01);
    \coordinate [right=3 of vm1] (v11);
    \draw[cb-red, line width = 3.5] (v01) circle(3);
    \draw[cb-red, line width = 3.5] (v11) circle(3);
    \filldraw (vm1) circle (3pt);

    \coordinate [below=4 of v] (vm2);
    \coordinate [left=3 of vm2] (v02);
    \coordinate [right=3 of vm2] (v12);
    \draw[cb-red, line width = 3.5] (v02) circle(3);
    \draw[cb-yellow, line width = 3.5] (v12) circle(3);
    \filldraw (vm2) circle (3pt);
\end{tikzpicture}%
} \;+\; \scalebox{.3}{\begin{tikzpicture}[x=1ex,y=1ex,baseline={([yshift=-1.5ex]current bounding box.center)}]
    \coordinate (vm);
    \coordinate [left=4 of vm] (v0);
    \coordinate [right=4 of vm] (v1);
    \coordinate [left=4 of v0] (vc1);
    \coordinate [right=4 of v1] (vc2);
    \draw[cb-red, line width = 3.5] (vc1) circle(4);
    \draw[cb-red, line width = 3.5] (vm) circle(4);
    \draw[cb-yellow, line width = 3.5] (vc2) circle(4);
    \filldraw (v0) circle (3pt);
    \filldraw (v1) circle (3pt);
\end{tikzpicture}%
} = \tfrac{1}{32} + \tfrac{1}{8} = \tfrac{5}{32}, \\
        c_2^{(2)} &= \scalebox{.3}{\begin{tikzpicture}[x=1ex,y=1ex,baseline={([yshift=-1.5ex]current bounding box.center)}]
    \coordinate (v);

    \coordinate [above=4 of v] (vm1);
    \coordinate [left=3 of vm1] (v01);
    \coordinate [right=3 of vm1] (v11);
    \draw[cb-red, line width = 3.5] (v01) circle(3);
    \draw[cb-red, line width = 3.5] (v11) circle(3);
    \filldraw (vm1) circle (3pt);

    \coordinate [below=4 of v] (vm2);
    \coordinate [left=3 of vm2] (v02);
    \coordinate [right=3 of vm2] (v12);
    \draw[cb-yellow, line width = 3.5] (v02) circle(3);
    \draw[cb-yellow, line width = 3.5] (v12) circle(3);
    \filldraw (vm2) circle (3pt);
\end{tikzpicture}%
}  + \scalebox{.3}{\begin{tikzpicture}[x=1ex,y=1ex,baseline={([yshift=-1.5ex]current bounding box.center)}]
    \coordinate (v);

    \coordinate [above=4 of v] (vm1);
    \coordinate [left=3 of vm1] (v01);
    \coordinate [right=3 of vm1] (v11);
    \draw[cb-red, line width = 3.5] (v01) circle(3);
    \draw[cb-yellow, line width = 3.5] (v11) circle(3);
    \filldraw (vm1) circle (3pt);

    \coordinate [below=4 of v] (vm2);
    \coordinate [left=3 of vm2] (v02);
    \coordinate [right=3 of vm2] (v12);
    \draw[cb-red, line width = 3.5] (v02) circle(3);
    \draw[cb-yellow, line width = 3.5] (v12) circle(3);
    \filldraw (vm2) circle (3pt);
\end{tikzpicture}%
} +  \scalebox{.3}{\begin{tikzpicture}[x=1ex,y=1ex,baseline={([yshift=-1.5ex]current bounding box.center)}]
    \coordinate (vm);
    \coordinate [left=4 of vm] (v0);
    \coordinate [right=4 of vm] (v1);
    \draw[cb-yellow, line width = 3.5] (v0) to[bend left=45] (v1);
    \draw[cb-yellow, line width = 3.5] (v0) to[bend right=45] (v1);
    \draw[cb-red, line width = 3.5] (vm) circle(4);
    \filldraw (v0) circle (3pt);
    \filldraw (v1) circle (3pt);
\end{tikzpicture}%
} + \scalebox{.3}{\begin{tikzpicture}[x=1ex,y=1ex,baseline={([yshift=-1.5ex]current bounding box.center)}]
    \coordinate (vm);
    \coordinate [left=4 of vm] (v0);
    \coordinate [right=4 of vm] (v1);
    \coordinate [left=4 of v0] (vc1);
    \coordinate [right=4 of v1] (vc2);
    \draw[cb-red, line width = 3.5] (vc1) circle(4);
    \draw[cb-yellow, line width = 3.5] (vm) circle(4);
    \draw[cb-red, line width = 3.5] (vc2) circle(4);
    \filldraw (v0) circle (3pt);
    \filldraw (v1) circle (3pt);
\end{tikzpicture}%
} + \scalebox{.3}{\begin{tikzpicture}[x=1ex,y=1ex,baseline={([yshift=-1.5ex]current bounding box.center)}]
    \coordinate (vm);
    \coordinate [left=4 of vm] (v0);
    \coordinate [right=4 of vm] (v1);
    \coordinate [left=4 of v0] (vc1);
    \coordinate [right=4 of v1] (vc2);
    \draw[cb-yellow, line width = 3.5] (vc1) circle(4);
    \draw[cb-red, line width = 3.5] (vm) circle(4);
    \draw[cb-yellow, line width = 3.5] (vc2) circle(4);
    \filldraw (v0) circle (3pt);
    \filldraw (v1) circle (3pt);
\end{tikzpicture}%
} = \tfrac{1}{64} + \tfrac{1}{32} + \tfrac{1}{8} + \tfrac{1}{16} + \tfrac{1}{16} = \tfrac{19}{64}, \\   
        c_2^{(3)} &= \scalebox{.3}{\begin{tikzpicture}[x=1ex,y=1ex,baseline={([yshift=-1.5ex]current bounding box.center)}]
    \coordinate (v);

    \coordinate [above=4 of v] (vm1);
    \coordinate [left=3 of vm1] (v01);
    \coordinate [right=3 of vm1] (v11);
    \draw[cb-red, line width = 3.5] (v01) circle(3);
    \draw[cb-yellow, line width = 3.5] (v11) circle(3);
    \filldraw (vm1) circle (3pt);

    \coordinate [below=4 of v] (vm2);
    \coordinate [left=3 of vm2] (v02);
    \coordinate [right=3 of vm2] (v12);
    \draw[cb-yellow, line width = 3.5] (v02) circle(3);
    \draw[cb-yellow, line width = 3.5] (v12) circle(3);
    \filldraw (vm2) circle (3pt);
\end{tikzpicture}%
} \;+\; \scalebox{.3}{\begin{tikzpicture}[x=1ex,y=1ex,baseline={([yshift=-1.5ex]current bounding box.center)}]
    \coordinate (vm);
    \coordinate [left=4 of vm] (v0);
    \coordinate [right=4 of vm] (v1);
    \coordinate [left=4 of v0] (vc1);
    \coordinate [right=4 of v1] (vc2);
    \draw[cb-red, line width = 3.5] (vc1) circle(4);
    \draw[cb-yellow, line width = 3.5] (vm) circle(4);
    \draw[cb-yellow, line width = 3.5] (vc2) circle(4);
    \filldraw (v0) circle (3pt);
    \filldraw (v1) circle (3pt);
\end{tikzpicture}%
} = \tfrac{1}{32} + \tfrac{1}{8} = \tfrac{5}{32}, \\
        c_2^{(4)} &= \scalebox{.3}{\begin{tikzpicture}[x=1ex,y=1ex,baseline={([yshift=-1.5ex]current bounding box.center)}]
    \coordinate (v);

    \coordinate [above=4 of v] (vm1);
    \coordinate [left=3 of vm1] (v01);
    \coordinate [right=3 of vm1] (v11);
    \draw[cb-yellow, line width = 3.5] (v01) circle(3);
    \draw[cb-yellow, line width = 3.5] (v11) circle(3);
    \filldraw (vm1) circle (3pt);

    \coordinate [below=4 of v] (vm2);
    \coordinate [left=3 of vm2] (v02);
    \coordinate [right=3 of vm2] (v12);
    \draw[cb-yellow, line width = 3.5] (v02) circle(3);
    \draw[cb-yellow, line width = 3.5] (v12) circle(3);
    \filldraw (vm2) circle (3pt);
\end{tikzpicture}%
} \; +\;  \scalebox{.3}{\begin{tikzpicture}[x=1ex,y=1ex,baseline={([yshift=-1.5ex]current bounding box.center)}]
    \coordinate (vm);
    \coordinate [left=4 of vm] (v0);
    \coordinate [right=4 of vm] (v1);
    \draw[cb-yellow, line width = 3.5] (v0) to[bend left=45] (v1);
    \draw[cb-yellow, line width = 3.5] (v0) to[bend right=45] (v1);
    \draw[cb-yellow, line width = 3.5] (vm) circle(4);
    \filldraw (v0) circle (3pt);
    \filldraw (v1) circle (3pt);
\end{tikzpicture}%
} \;+\; \scalebox{.3}{\begin{tikzpicture}[x=1ex,y=1ex,baseline={([yshift=-1.5ex]current bounding box.center)}]
    \coordinate (vm);
    \coordinate [left=4 of vm] (v0);
    \coordinate [right=4 of vm] (v1);
    \coordinate [left=4 of v0] (vc1);
    \coordinate [right=4 of v1] (vc2);
    \draw[cb-yellow, line width = 3.5] (vc1) circle(4);
    \draw[cb-yellow, line width = 3.5] (vm) circle(4);
    \draw[cb-yellow, line width = 3.5] (vc2) circle(4);
    \filldraw (v0) circle (3pt);
    \filldraw (v1) circle (3pt);
\end{tikzpicture}%
} = \tfrac{1}{128} + \tfrac{1}{48} + \tfrac{1}{16} = \tfrac{35}{384}. \qedhere
    \end{align*}
\end{example}

\begin{remark}[Ising model]
\label{rmk:ising}
Our examples are motivated from the physical Ising model as we will briefly explain in this remark.
The partition function of the \emph{critical Ising model} on a specific monochromatic graph $G$ is defined by
$$
Z(G,\lambda) = \sum_{\substack{\gamma \subset G\\\gamma \text{ Eulerian}}} \lambda^{|E(\gamma)|},
$$
where we sum over all Eulerian subgraphs $\gamma$ of $G$ (see, e.g.,~\cite{MR2989454}). This means that if we delete all edges of $G$ that are not in $\gamma$, then the resulting graph shall only have vertices of even degree.
A pair $(G,\gamma)$ of a monochromatic graph $G$ and an Eulerian subgraph $\gamma \subset G$ is equivalent to an edge-bicolored graph in which an even number of \colT{} edges belongs to each vertex.

Notice that we effectively designed the polynomial $g(x,y)$ from Example~\ref{ex:Ising_model_deg4_part1} and equivalently the coefficients $\Lambda_{u,w}$, such that  the coefficient of $\lambda^k$ in $A_n$ is the
automorphism-weighted number of 4-regular graphs with $k$ \colT{} edges where an even number of \colT{} edges belong to each vertex.

Hence, with $A_n$ as defined in Example~\ref{ex:Ising_model_deg4_part1}, we find that
$$
A_n = \sum_{\substack{G}} \frac{Z(G,\lambda)}{|\Aut G|},
$$
where we sum over all monochromatic graphs $G$ that are $4$-regular 
and which have Euler characteristic $-n$. We can thus 
interpret $A_n$ as the partition function of the critical Ising 
model of a \emph{random} 4-regular monochromatic graph of fixed Euler characteristic. 
Here, random means that each monochromatic graph $G$ is sampled with probabilty $1/|\Aut G|$.
\end{remark}

\section[Efficient computation of the coefficients An]{Efficient computation of the coefficients $A_n$}
\label{sec:algo}

In this section, we describe an effective algorithm to compute the coefficients $A_n$ 
that encode the asymptotic expansion of the integral~\eqref{equ:exponentialIntegral}, 
and the weighted numbers of edge-bicolored graphs of Euler characteristic $-n$, 
by Theorem~\ref{thm:laplace-graphs}.
The algorithm is implemented in \texttt{Julia} and is available at \cite{mathrepo}.

\begin{proposition}\label{prop:algo}
For a given integer $n\geq 1$, and
the coefficients $\Lambda_{u,w}$ as required by Theorem~\ref{thm:laplace-graphs},
the following algorithm correctly computes $A_0,\ldots,A_n$:
\setdefaultleftmargin{0pt}{}{}{}{}{}
\begin{enumerate}
\item[\underline{Step 1}:]
Define the polynomials 
$$F_k(x,y) = 
\sum_{\substack{u,w\geq0\\u+w=k+2}} \Lambda_{u,w}\frac{x^u y^w}{u!w!} \text{ for } k \in \{1,\ldots, 2n\};
$$
\item[\underline{Step 2}:]
Set $Q_0(x,y)=1$ and recursively compute $Q_1,\ldots, Q_{2n}$ using
$$
Q_{m}(x,y) = \frac{1}{m} \sum_{k=1}^{m} k F_k(x,y) Q_{m-k}(x,y) \text{ for } m \in\{1,\ldots,2n\};
$$
\item[\underline{Step 3}:]
Let $q^{(k)}_{s,t}$ be the coefficients of $Q_{k}(x,y) = \sum_{s,t\geq 0} q^{(k)}_{s,t} x^sy^t$.
Then, 
$$
A_k = \sum_{s,t \geq 0} (2s-1)!! \cdot (2t-1)!! \cdot q^{(2k)}_{2s,2t} \text{ for } k \in \{0,\ldots,n\}.
$$
\end{enumerate}
\end{proposition}
To run the algorithm with a fixed $n$, it is sufficient to know $\Lambda_{u,w}$ for all $u,w \geq 0$ with $u+w \leq 2n+2$. Also, recall that we require $\Lambda_{u,w}=0$ if $u+w <3$.

\begin{proof}
    By Proposition~\ref{prop:laplace}, we have this identity of power series in $z^{-n}$:
    \[
        \sum_{n\geq 0} A_n z^{-n} = \sum_{s,t\geq 0} z^{-(s+t)}\cdot(2s-1)!!\cdot (2t-1)!!\cdot a_{2s,2t}(z \cdot \Lambda),
    \]
where $a_{s,t}(z \cdot \Lambda)$ is as described in Proposition~\ref{prop:laplace} and \eqref{eq:coeffs_extraction_gen_funct}:
$$
\sum_{s,t \geq 0} 
 a_{s,t}(z \cdot \Lambda) x^s y^t = 
\exp
\left(
z
\sum_{\substack{u,w \geq 0\\ u+w \geq 3}} \Lambda_{u,w} \frac{x^uy^w}{u!w!}
\right).
$$
Rescaling $(x,y) \mapsto (x/\sqrt{z},y/\sqrt{z})$ in the above formula gives the following identity of 
power series in $\R[x,y][[z^{-1/2}]]$,
\begin{align}
\label{eq:aF}
\sum_{s,t \geq 0} 
 z^{-\frac{s+t}{2}} a_{s,t}(z \cdot \Lambda) x^s y^t = 
\exp
\left(
\sum_{k \geq 1} z^{-\frac{k}{2}}
F_k(x,y)
\right),
\end{align}
where we used the definition of $F_k(x,y)$ in the statement.
Let $q_{s,t}^{(k)}$ be the coefficients 
$
\sum_{k \geq 0} q_{s,t}^{(k)} z^{-\frac{k}{2}} = 
z^{-\frac{s+t}{2}} a_{s,t}(z \cdot \Lambda)$.
With this definition, \eqref{eq:aF} and the formula under Step 3 in the statement
correctly compute $A_k$.

It remains to prove that the coefficients $q_{s,t}^{(k)}$ are computed correctly by Step 2 in the statement.
Rewrite \eqref{eq:aF} using the definition of $Q_k(x,y)$, before applying the derivative operator $z\frac{\partial}{\partial z}$ on both sides. This gives
\begin{align*}
z\frac{\partial}{\partial z}\left(\sum_{k \geq 0}Q_k(x,y) z^{-\frac{k}{2}}\right)
&=
z\frac{\partial}{\partial z}\exp\left(\sum_{k \geq 1} z^{-\frac{k}{2}}F_k(x,y)\right)\\
\Rightarrow
-\sum_{k \geq 0}\frac{k}{2}Q_k(x,y) z^{-\frac{k}{2}}
&=
-\left(\sum_{m \geq 0}Q_m(x,y) z^{-\frac{m}{2}}\right)\sum_{k \geq 1} \frac{k}{2}z^{-\frac{k}{2}}F_k(x,y).
\end{align*}
The recursive relation between $Q_m$ and $F_k$ follows
by comparing the $z^{-\frac{m}{2}}$ coefficients on both sides of this equation.
\end{proof}

\section{Asymptotics and critical points}
\label{sec:asymptotics}

In this section, we study the asymptotic behavior of the coefficients 
$A_n$ in Theorem~\ref{thm:laplace-graphs} for large $n$. 
Here, we will restrict ourselves to 
\emph{regular} edge-bicolored graphs, 
meaning that each vertex has a fixed degree $k \geq 3.$
For fixed coefficients $\ol_{u,w}$ given for $u,w\geq 0$ with $u+w=k$, we study the weighted sum over graphs
\[
A_n = \sum_{G \in \GG_{-n}^k} 
        \frac{1}{
\left\lvert \Aut(G) \right\rvert} 
\prod \limits _{v\in V^G}
\ol_{\deg(v)},
\]
where $\GG_{-n}^k$ is the set of all regular (edge-bicolored) graphs with vertex degree $k$
and Euler characteristic $-n=|V^G|-|E^G|$.
As for each $k$-regular graph $G$ we have $k|V^G| = 2|E^G|$,
all graphs in $\GG_{-n}^k$ have $\frac{2n}{k-2}$ vertices and 
$\frac{n k}{k-2}$ edges.
It is convenient to define the homogeneous polynomial
$$V(x,y) = \sum_{\substack{u,w \geq 0\\ u+w = k}} \ol_{u,w} \frac{x^u y^w}{u!w!} \in \R[x,y].$$
Let $\Phi$ be the set of global maxima of the function 
\[
    S^1=  \{ (x,y) \in \R^2 \,:\, x^2+y^2=1 \} \rightarrow \R_{\geq 0},\quad (x,y) \mapsto |V(x,y)|.
\]
A point $(x,y) \in \Phi$ is \emph{non-degenerate} if
$k^2V(x,y) \neq \left( \frac{\partial^2 V}{\partial x^2}(x,y) + \frac{\partial^2 V}{\partial y^2}(x,y) \right)$. 

\begin{proposition}
\label{prop:AnAsy}
Let $M = \frac{k}{k-2}$ and $K=\frac{2}{k-2}$.
If $A_n$, $V$, $\ol_{u,w}$ and $\Phi$ are related as described above and all extrema in $\Phi$ are non-degenerate, then
$$
A_n \sim 
\begin{cases}
\frac{1}{2\sqrt{2}\pi}k^{nM+\frac{1}{2}}K^{n-\frac{1}{2}}\Gamma(n) 
\sum\limits_{(x,y) \in \Phi} \frac{V(x,y)^{nK}}{\sqrt{B(x,y)}}
&\text{ if } nK, nM \in \Z,
\\
0 &\text{ else,}
\end{cases}
$$
where
$$
B(x,y) = 
k^2 - 
\frac{ \frac{\partial^2 V}{\partial x^2}(x,y) + \frac{\partial^2 V}{\partial y^2}(x,y)}{ V(x,y) } \quad \text{ for } (x,y)\in S^1.
$$
\end{proposition}

We will prove this theorem by first proving an 
integral representation of the coefficients $A_n$.
Afterwards, we will apply the Laplace method in its 
 one dimensional form to provide an asymptotic expression for this 
integral in the large $n$ limit.

\begin{lemma}
\label{lmm:AnHomo}
Let $M = \frac{k}{k-2}$ and $K=\frac{2}{k-2}$.
For a given integer $n \geq 0$ such that $nK$ and $nM$ are integers, we have
\begin{gather*}
A_n = \frac{2^{nM} (nM)! }{ 2\pi \cdot (nK)!} 
\int_{-\pi}^\pi  V(\cos\varphi ,\sin \varphi)^{nK} \dd \varphi.
\end{gather*}
If $nK$ or $nM$ is not an integer then $A_n =0$.
\end{lemma}
\begin{proof}
A $k$-regular graph
has $nM$ edges and $nK$ vertices, so
$nM$ and $nK$ must be integers; otherwise $A_n =0$. We
will assume the former.
By Proposition~\ref{prop:genfun2colgraphs},
$$
A_n = \sum_{\substack{s,t\geq 0\\ s+t=nM}} 
(2s-1)!! \cdot (2t-1)!! \cdot a_{2s,2t}(\ol).
$$
If $s+t = nM$, it follows from \eqref{eq:coeffs_extraction_gen_funct} that $a_{2s,2t}(\Lambda)$ is a homogeneous polynomial  of degree $nK$. Because $\exp(X) = \sum_{N \geq 0} \frac{X^N}{N!}$, it also follows that $a_{2s,2t}(\Lambda)$ is the coefficient in front of $x^{2s}y^{2t}$ in the quotient $V(x,y)^{nK}/(nK)!$, since $V$ is homogeneous.
Using $\frac{1}{\sqrt{ 2\pi} }\int_\R e^{-\frac{x^2}{2}} x^{2s}\dd x = (2s-1)!!$ and $\int_\R e^{-\frac{x^2}{2}} x^{2s+1}\dd x = 0$ for integers $s$, we obtain 
\begin{gather*}
A_n = \frac{1}{2\pi \cdot (nK)!} 
\int_{\R^2} e^{-\frac{x^2}{2} - \frac{y^2}{2} }\, V(x,y)^{nK} \dd x \dd y.
\end{gather*}
We can pass to polar coordinates and use $V(rx,ry) = r^k V(x,y)$ together with 
$$
\int_0^\infty e^{-\frac{r^2}{2}} r^{nkK+1} \dd r =
\int_0^\infty e^{-q} (2q)^{nkK/2} \dd q = 2^{nM} (nM)!
$$
to prove the lemma.
\end{proof}

\begin{proof}[Proof of Proposition~\ref{prop:AnAsy}]
We are interested in the cases in which $A_n\neq 0$.
When $n$ is large, the main contribution to the integral 
in the statement of Lemma~\ref{lmm:AnHomo} comes 
from angles $\varphi$ where $|V(\cos \varphi,\sin \varphi)|$ is 
maximal. Let $\varphi_c$ be the location of such a maximum.
By definition, we have $(\cos \varphi_c,\sin \varphi_c) \in \Phi$.
Near this maximum, we get the Taylor expansion
$$
f_{\varphi_c}(\varphi):=
\log \frac{V(\cos \varphi, \sin \varphi)}{V(\cos \varphi_c, \sin \varphi_c)}
=
-
B(\cos \varphi_c, \sin \varphi_c)
\frac{(\varphi -\varphi_c)^2}{2}
+\bigO((\varphi -\varphi_c)^3),
$$
where $B(\cos \varphi_c, \sin \varphi_c)$ is defined as in the statement.
Because $\varphi_c$ is a maximum of $|V(\cos \varphi_c, \sin \varphi_c)|$, 
we have $B(\cos \varphi_c, \sin \varphi_c) \geq 0$. Our assumption that all the maxima are non-degenerate hence implies that $B(\cos \varphi_c, \sin \varphi_c) > 0$.\par 

We may therefore write, for some sufficiently small $\varepsilon > 0$,
\begin{gather*}
A_n = \frac{2^{nM} (nM)! }{ 2\pi \cdot (nK)!} 
\sum_{(\cos \varphi_c,\sin \varphi_c) \in \Phi}
\!\!\!V(\cos \varphi_c, \sin \varphi_c)^{nK}
\int_{\varphi_c-\varepsilon}^{\varphi_c+\varepsilon} e^{nK f_{\varphi_c}(\varphi)} \dd \varphi + R_1(n,\varepsilon).
\end{gather*}
From the Taylor expansion of the function $f_{\varphi_c}(\varphi)$ and
Lemma~\ref{lmm:AnHomo},
it follows by the same reasoning as in the proof of Proposition~\ref{prop:laplace} that the remainder term respects the bound
$|R_1(n,\varepsilon)| \leq C_1 \exp({-C_2 n \varepsilon^2})$
with some constants $C_1,C_2>0$.
Specifying $\varepsilon = n^{-\gamma}$ with $\gamma \in (\frac13,\frac12)$ allows to truncate the Taylor expansion of $f_{\varphi_c}$ after the second term without changing asymptotic behavior in the $n\rightarrow \infty$ limit. Hence,
$$
\int_{\varphi_c-\varepsilon}^{\varphi_c+\varepsilon} e^{nK f_{\varphi_c}(\varphi)} \dd \varphi =
\int_{-\varepsilon}^{\varepsilon} \exp\left(-nK \frac{B(\cos\varphi_c,\sin \varphi_c)}{2} \varphi^2\right) \dd \varphi + R_2(n, \varepsilon).
$$
The remainder term fulfills $|R_2(n, \varepsilon)| < C_3 n^{\frac12} \varepsilon^3$ for some $C_3>0$.
Again, as in the proof of Proposition~\ref{prop:laplace}, we may complete the Gaussian integral to find that 
$$
\int_{\varphi_c-\varepsilon}^{\varphi_c+\varepsilon} e^{nK f_{\varphi_c}(\varphi)} \dd \varphi = \sqrt{ \frac{2\pi}{nK \cdot B(\cos\varphi_c,\sin \varphi_c)}} + \bigO (n^{-1}).
$$
The result follows from Stirling's formula $\Gamma(n) \sim \sqrt{2\pi n^{-1}} n^n e^{-n}$ as $n \rightarrow \infty$.
\end{proof}

\begin{example}\label{ex:Ising_critVsphere}
    To continue the running example of the Ising model (cf.~Example~\ref{ex:Ising_model_deg4_part2}), let $V(x,y) = \frac{x^4}{4!} + \lambda \frac{x^2 y^2}{4} + \lambda^2 \frac{y^4}{4!}$.
    We want to find the critical points of $V$ on the circle, that means the points $(x,y)\in\R^2$ with $x^2+y^2=1$ satisfying
    \begin{equation}\label{eq:sys_sphere}
    y \frac{\partial V}{\partial x}(x,y) = x \frac{\partial V}{\partial y}(x,y).
    \end{equation}
    We get the following eight critical points:
    \begin{equation}\label{eq:critV}
     (\pm 1, 0), (0,\pm 1), \left( \pm \frac{\sqrt{\lambda(\lambda-3)}}{\sqrt{\lambda^2-6 \lambda+1}}, \pm \frac{\sqrt{1-3\lambda}}{\sqrt{\lambda^2-6 \lambda+1}} \right) \in \R^2.
    \end{equation}
    \begin{figure}[ht]
\ifkeepslowthings
        \centering
        \input{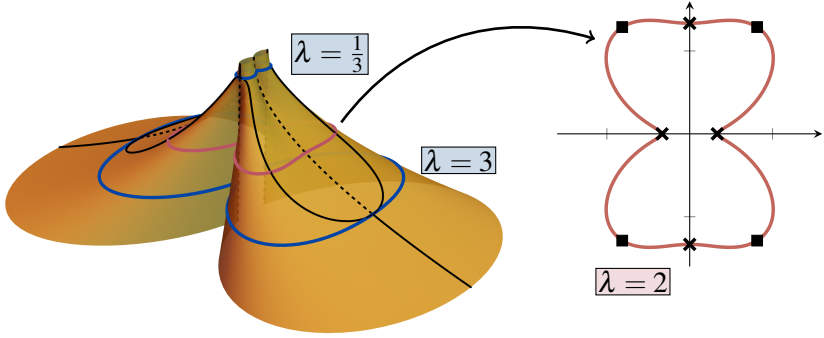}
\fi
        \caption{The system \eqref{eq:sys_sphere} for the function $V$ from Example \ref{ex:Ising_critVsphere}. Left: all values of $\lambda \in (0,4)$ on the (reversed) vertical axis. At each level $\lambda=\text{const.}$ the black continuous curves are the maxima; the dashed curves are the minima. In blue, the curves for $\lambda=\frac{1}{3}$ and $\lambda=3$ where the behavior of the maxima changes.
        Right: the section $\lambda=2$, with its maxima (squares) and its minima (crosses).
        }
        \label{fig:sing_various_lambdas_3D}
    \end{figure}
Our case of interest is $\lambda >0$. Then, the last $4$-tuple of points is real if and only if $\lambda\in [\frac{1}{3}, 3]$, and in that interval those are the maxima of $|V|$ on $S^1$. For $\lambda<\frac{1}{3}$, the maxima are $(\pm 1, 0)$, whereas for $\lambda>3$, the maxima are $(0,\pm 1)$. Figure \ref{fig:sing_various_lambdas_3D} displays the function $(V(x,y) x, V(x,y) y)$ and its critical points, for $\lambda\in (0,4)$.

We can now use Proposition \ref{prop:AnAsy} to find
$A_n \sim c \:\Gamma(n) \alpha^n$,
where $c = c(\lambda)$ and $\alpha = \alpha(\lambda)$ are piecewise defined as
\begin{center}
    {\def\arraystretch{1.8}
    \begin{tabular}{c|c|c}
         & $\alpha(\lambda)$ & $c(\lambda)$ \\
        \hline
       $0<\lambda<\tfrac{1}{3}$  & $\frac{2}{3}$ & $\frac{1}{\pi} \sqrt{\frac{1}{2-6\lambda}}$ \\
       $\tfrac{1}{3}<\lambda<3$  & $\frac{-16\lambda^2}{3\lambda^2-18\lambda+3}$ & $\frac{1}{\pi} \sqrt{\frac{8\lambda}{-3\lambda^2+10\lambda -3}}$ \\
       $\lambda>3$  & $\frac{2\lambda^2}{3}$ & $\frac{1}{\pi} \sqrt{\frac{\lambda}{2\lambda -6}}$ \\
    \end{tabular}
    }
\end{center}
The function $\alpha$ is continuous, it is not $C^1$-differentiable at $\lambda=\frac{1}{3}$, and it is $C^1$- but not $C^2$-differentiable at $\lambda=3$. On the other hand, the limits of $c(\lambda)$ at $\frac{1}{3}$, $3$ go to infinity from both sides.
This can be observed in Figure \ref{fig:Ising_deg4_alpha}.
The points $\lambda=\frac13$ and $\lambda=3$ where the functions $\alpha(\lambda)$ and $c(\lambda)$ are non-analytic 
are \emph{phase transition} points.
Phase transitions are of pivotal interest in statistical physics. Here, we find the phase transitions of the Ising model on a random 4-regular graph. In each of the three regions for the parameter $\lambda$, the statistical system is expected to exhibit intrinsically different behaviors.

Note that using Proposition \ref{prop:algo} we can also compute $A_n$ for large $n$ and solve for $\alpha$ and $c$ numerically. For details see our implementation at \cite{mathrepo}.
    \begin{figure}[!ht]
        \centering
\ifkeepslowthings
        \includegraphics[width = 0.45\textwidth]{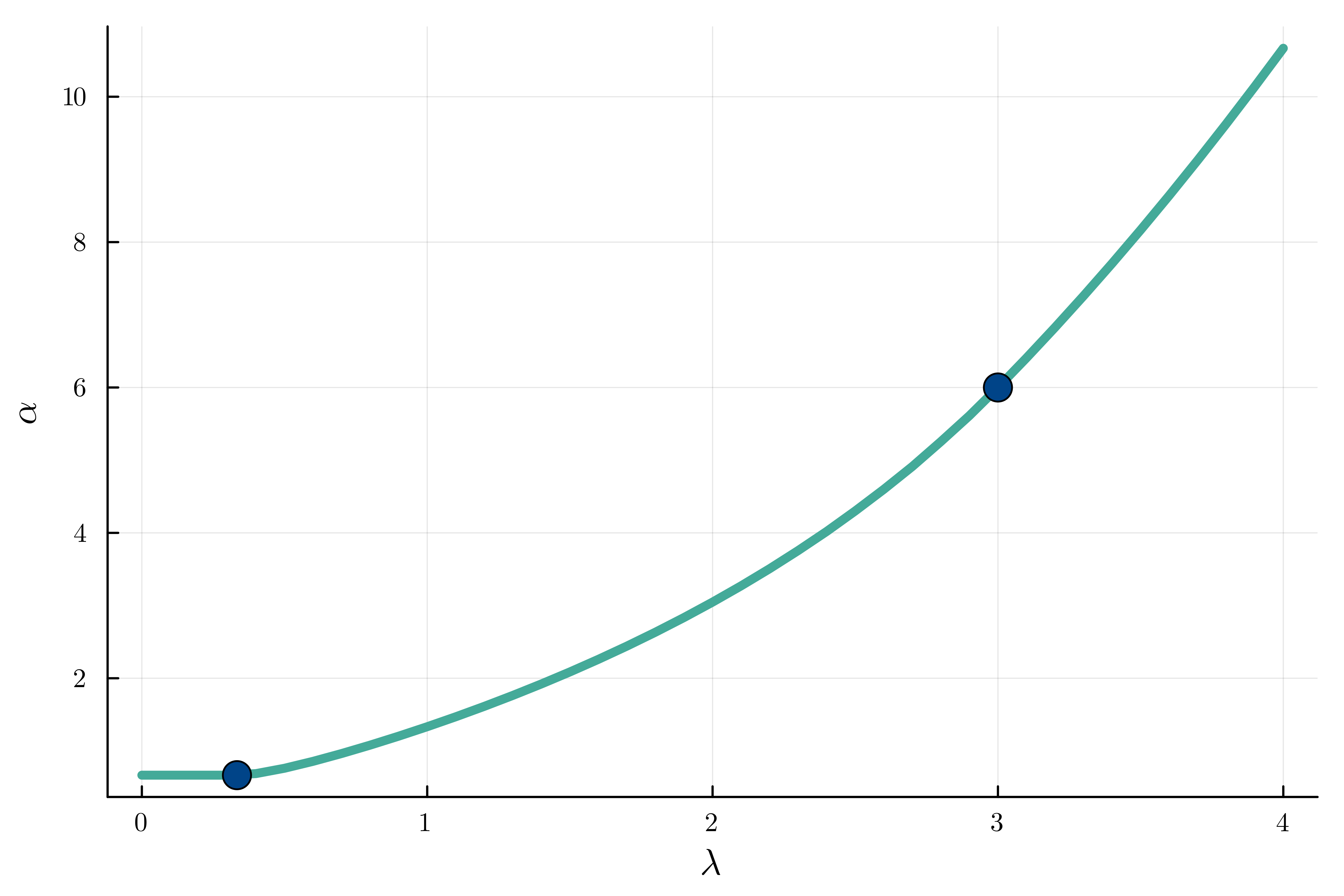}
        \quad
        \includegraphics[width = 0.45\textwidth]{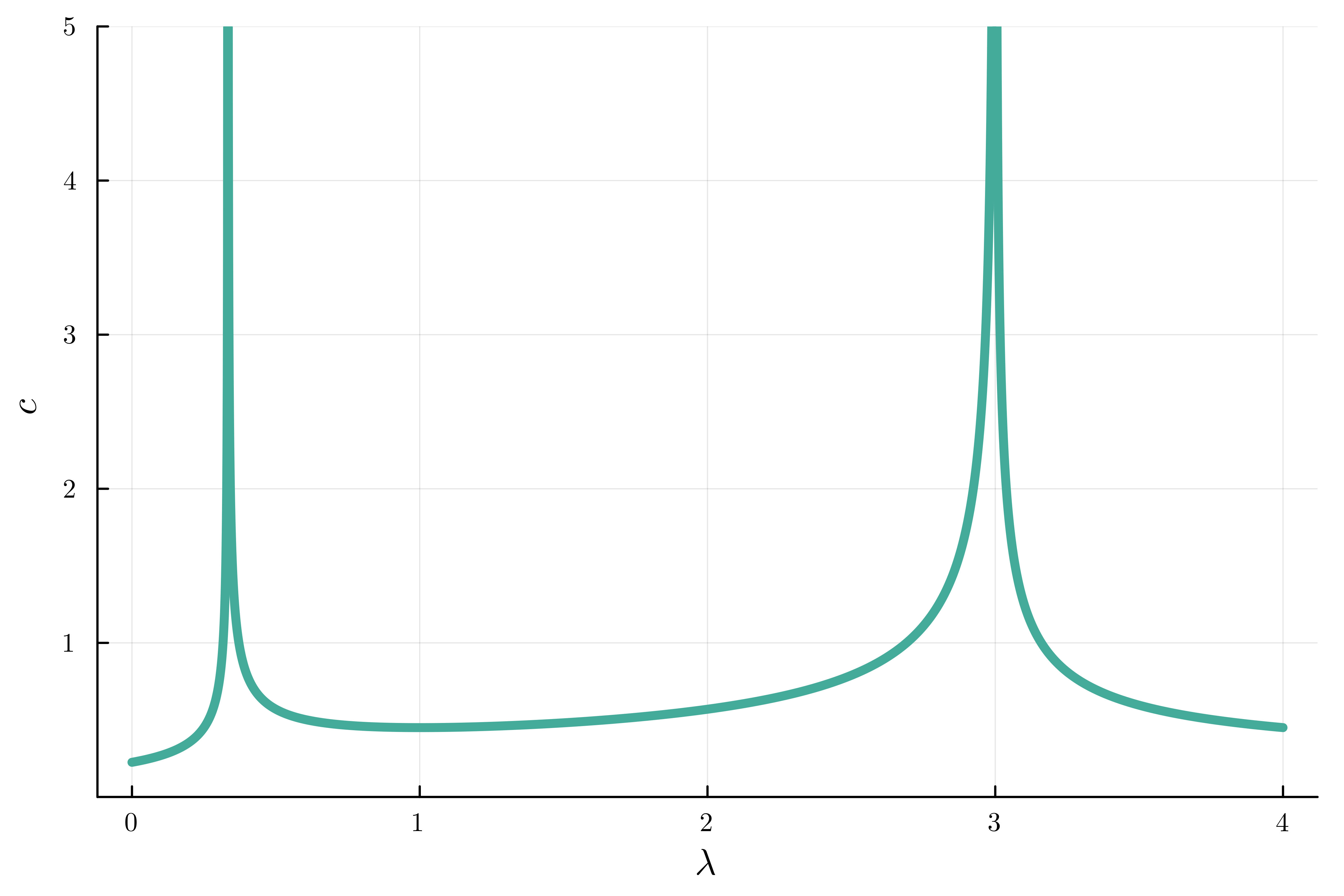} 
\fi
        \caption{The behavior of $\alpha(\lambda)$ and $c(\lambda)$ in the Ising model from Example \ref{ex:Ising_critVsphere}. 
        The phase transitions at $\lambda=\frac{1}{3}, 3$ can be detected in both quantities.}
        \label{fig:Ising_deg4_alpha}
    \end{figure}
\end{example}

It is common belief in physics that the asymptotic behavior of $A_n$ depends on the critical points of $g(x,y) = -\frac{x^2}{2}-\frac{y^2}{2} + V(x,y)$ (see, e.g.,~\cite{le2012large}). Moreover, it is well-known, also in applied mathematics, that identifying the critical point that contributes most to the asymptotics is a complicated \emph{connection problem} \cite{berry1990hyperasymptotics}.
With this in mind, we rephrase Proposition \ref{prop:AnAsy} in terms of the critical points of $g$ instead of those of $V$ restricted to the sphere. We write $\operatorname{crit}_D f$ for the set of critical points of $f$ restricted to the domain $D$. Let
\[
\Psi = \left\{ (w,z) \in \crit_{\C\cdot\R^2}g\!\setminus\! \{\mathbf{0}\} \,:\, \|(w,z)\|\leq \|(w',z')\| \, \forall (w',z') \in \crit_{\C\cdot\R^2}g\!\setminus\! \{\mathbf{0}\} \right\},
\]
where $\C\cdot\R^2$ is the set of complex points $(w,z)$ whose ratio (when well-defined) is real.
We call points in $\Psi$ \emph{non-degenerate} if the Hessian (the matrix of second derivatives) of $g$ has full rank.

\begin{theorem}\label{thm:crit_g}
    Assume that $A_n$, $g$, and $\Psi$ are related as described above and all extrema in $\Psi$ are non-degenerate. Then 
    \begin{equation}
    \label{eq:niceFormula}
    A_n \sim 
    \frac{1}{2\pi}
    \Gamma(n)
    \sum\limits_{(w,z) \in \Psi} \frac{(-g(w,z))^{-n}}{\sqrt{-\det H_g(w,z)}}.
    \end{equation}
\end{theorem}
\begin{proof}
    Our goal is to express $A_n$ from Proposition \ref{prop:AnAsy} in terms of the critical points of $g$. The first step is to associate the critical points of $V$ to those of $g$. 
    Given $(x,y)\in\crit_{S_1}(V)$, we look for some $\ell \in \C^*$ such that $(\ell x, \ell y)\in\operatorname{crit}_{\C\cdot\R^2}(g)$. Imposing the conditions $\ell x = \frac{\partial V}{\partial x}(\ell x,\ell y)$, $\ell y = \frac{\partial V}{\partial y}(\ell x,\ell y)$, and using homogeneity of $V$, we get 
    \begin{equation}
        \label{eq:ellRoots}
        \ell^{2-k} = k V(x,y).
    \end{equation}
    Therefore, as $k \geq 3$,
    \[
    \max_{(x,y)\in \operatorname{crit}_{S^1}(V)} V(x,y) = \frac{1}{k} \left( \min_{(w,z)\in \operatorname{crit}_{\C^2}(g)} \|(w,z)\| \right)^{2-k},
    \]
    so every element $(w,z)\in\Psi\subset \C\cdot\R^2$ arises as $(\ell x,\ell y)$ for some $(x,y)\in\Phi$.

    Using these considerations, we write the result from Proposition \ref{prop:AnAsy} in terms of the critical points of $g$. 
    At a point $(w,z) = (\ell x, \ell y)\in \Psi$, by \eqref{eq:ellRoots}, we have
    \begin{equation}
    \label{eq:gAsEll}
    g(w,z) = -\frac{\ell^2}{2} + V(w,z) = -\frac{\ell^2}{2} + \ell^k V(x,y) = \ell^2 \frac{2-k}{2k}.
    \end{equation}
    Let $K = \frac{2}{k-2}$ and $M=\frac{k}{k-2}$. Then, we have
    \[
    V(x,y)^{nK} = k^{-nK} \ell^{-2n} = k^{-nK} (-kKg(w,z))^{-n} = k^{-nM} K^{-n} (-g(w,z))^{-n}.
    \]
    This allows to cancel pre-factors in the asymptotic expression for $A_n$ from Proposition \ref{prop:AnAsy}. 
    We are left to rewrite $B$ in terms of $(w,z)$. 
    Notice that the determinant of the Hessian of $g(w,z)$ can be expressed, using \eqref{eq:ellRoots}, as 
    \begin{align*}
    \det H_g(w,z) &= \frac{1}{\ell^2} \det H_g(\ell x, \ell y) = B(x,y)\ell^{k-2} V(x,y) (1-k(k-1)\ell^{k-2}V(x,y))\\
    &= -B(x,y)\frac{k-2}{k},
    \end{align*}
    where $(\ell, x, y)$ are new coordinates on $\C^*\times S^1$, and $(w,z) \in \Psi$.
    Hence,
    \begin{align*}
        A_n &\sim
        \frac{1}{2\sqrt{2}\pi}k^{nM+\frac{1}{2}}K^{n-\frac{1}{2}}\Gamma(n) 
        \sum\limits_{(x,y) \in \Phi} \frac{V(x,y)^{nK}}{\sqrt{B(x,y)}} \\
        &= \frac{1}{4\pi} \sqrt{k(k-2)} \Gamma(n) \frac{2}{k-2} \sum\limits_{(w,z) \in \Psi} \frac{(-g(w,z))^{-n}}{\sqrt{-\frac{k \det H_g(w,z)}{k-2}}}\\
        & = \frac{1}{2\pi} \Gamma(n) \sum\limits_{(w,z) \in \Psi} \frac{(-g(w,z))^{-n}}{\sqrt{-\det H_g(w,z)}},
    \end{align*}
    where the factor $\frac{2}{k-2}$ appears since each of the $k-2$ points $\{(w,z) = (\ell x, \ell y)\}$ in $\Psi$ is counted twice by the corresponding points $\{(x,y),(-x,-y)\}\in\Phi$.
\end{proof}

\begin{remark}
    The formula \eqref{eq:niceFormula} yields $0$ if $nM$ or $nK$ are not integers. Indeed, using \eqref{eq:gAsEll} from the proof above, we can write, for $(w,z)\in\Psi$, 
    \[
    g(w,z) = (l\cdot \zeta_i)^2 \frac{2-k}{2k}, \quad i\in\{1,...,k-2\},
    \]
    where $l\in\R$ and $\zeta_i$ is a $(k-2)$th root of unity, so $(w,z) = (l\zeta_ix,l\zeta_iy)$ for some $(x,y)\in\Phi$. Also $(l\zeta_jx,l\zeta_jy)\in\Psi$ for all $j\in\{1,\dots,k-2\}$. Therefore, the sum in \eqref{eq:niceFormula} becomes
    \begin{align*}
        \sum\limits_{(w,z) \in \Psi} \frac{(-g(w,z))^{-n}}{\sqrt{-\det H_g(w,z)}}
        & \propto 
        \sum_{j=1}^{k-2} \frac{\left( l\cdot\zeta_j \right)^{-2n} \left( \frac{k-2}{k}\right)^{-n} }{\sqrt{-\det H_g(l\zeta_j x, l\zeta_j y)}}\\
        & = \left\{ \begin{array}{ll}
            (k-2) \frac{l^{-2n} \left( \frac{k-2}{k}\right)^{-n}}{\sqrt{-\det H_g(l x, l y)}}& \text{if } (k-2)\mid 2n, \\
            0 & \text{else.}
        \end{array}\right.
    \end{align*}
    The condition $(k-2)\mid 2n$ is equivalent to $nK\in\Z$, which also implies $nM\in\Z$.
\end{remark}

\begin{figure}
    \centering
\ifkeepslowthings
    \begin{tikzpicture}
  \begin{axis}[
    width=1.8in,
    height=2.5in,
    axis lines=center,
    xlabel=\empty,
    ylabel=\empty,
    ylabel style={above right},
    xlabel style={below right},
    xtick={0,-3,3},
    ytick={0,-11,11},        
    xmin=-3.7,xmax=3.7,ymin=-12,ymax=12]
    \addplot+[no markers, samples=200, samples y=0, domain=-sqrt(6):sqrt(6), variable=\t, line width=1.3, color = cb-green]
    ({t}, {(2*sqrt(6-t^2)/sqrt(3)});
    \addplot+[no markers, samples=200, samples y=0, domain=-sqrt(6):sqrt(6), variable=\t, line width=1.3, color = cb-green]
    ({t}, {-(2*sqrt(6-t^2)/sqrt(3)});
    \draw[cb-green, line width = 1.3] (0,-11.8) -- (0, 11.6);

    \addplot+[no markers, samples=200, samples y=0, domain=-sqrt(8):sqrt(8), variable=\t, line width=1.3, color = cb-purple]
    ({t}, {2*sqrt(3)*sqrt(8 - t^2)});
    \addplot+[no markers, samples=200, samples y=0, domain=-sqrt(8):sqrt(8), variable=\t, line width=1.3, color = cb-purple]
    ({t}, {-2*sqrt(3)*sqrt(8 - t^2)});
    \draw[cb-purple, line width = 1.3] (-3.6,0) -- (3.4,0);

    \addplot [only marks, mark=x, line width = 1.3, mark size = 3pt] table {
    0  0
    0  9.79796
    0  -9.79796
    };
    \addplot [only marks, mark=square*, mark size = 2pt] table {
    2.44949  0
    -2.44949  0
    };

    \node at (2.3,-10.5) [rectangle,draw,fill = cb-palegrey,inner sep=0.8pt] {$\lambda = \frac{1}{4}$};
    \end{axis}
\end{tikzpicture}
\,
\begin{tikzpicture}
  \begin{axis}[
    width=1.8in,
    height=2.5in,
    axis lines=center,
    xlabel=\empty,
    ylabel=\empty,
    ylabel style={above right},
    xlabel style={below right},
    xtick={0,-3,3},
    ytick={0,-6,6},        
    xmin=-3.7,xmax=3.7,ymin=-7,ymax=7]
    \addplot+[no markers, samples=200, samples y=0, domain=-sqrt(6):sqrt(6), variable=\t, line width=1.3, color = cb-green]
    ({t}, {sqrt(2/3)*sqrt(6 - t^2)});
    \addplot+[no markers, samples=200, samples y=0, domain=-sqrt(6):sqrt(6), variable=\t, line width=1.3, color = cb-green]
    ({t}, {-(sqrt(6-t^2)*sqrt(2/3)});
    \draw[cb-green, line width = 1.3] (0,-6.9) -- (0, 6.7);

    \addplot+[no markers, samples=200, samples y=0, domain=-sqrt(4):sqrt(4), variable=\t, line width=1.3, color = cb-purple]
    ({t}, {sqrt(6*(4 - t^2))});
    \addplot+[no markers, samples=200, samples y=0, domain=-sqrt(4):sqrt(4), variable=\t, line width=1.3, color = cb-purple]
    ({t}, {-sqrt(6*(4 - t^2))});
    \draw[cb-purple, line width = 1.3] (-3.6,0) -- (3.4,0);

    \addplot [only marks, mark=x, line width = 1.3, mark size = 3pt] table {
    0  0
    2.44949  0
    -2.44949  0
    0  4.89898
    0  -4.89898
    };
    \addplot [only marks, mark=square*, mark size = 2pt] table {
    1.93649  1.22474
    -1.93649  1.22474
    1.93649  -1.22474
    -1.93649  -1.22474
    };

    \node at (2.3,-5.92) [rectangle,draw,fill = cb-palegrey,inner sep=0.8pt] {$\lambda = \frac{1}{2}$};
    \end{axis}
\end{tikzpicture}
\,
\begin{tikzpicture}
  \begin{axis}[
    width=1.8in,
    height=2.5in,
    axis lines=center,
    xlabel=\empty,
    ylabel=\empty,
    ylabel style={above right},
    xlabel style={below right},
    xtick={0,-3,3},
    ytick={0,-1,1},        
    xmin=-3.7,xmax=3.7,ymin=-1.5,ymax=1.5]
    \addplot+[no markers, samples=200, samples y=0, domain=-sqrt(6):sqrt(6), variable=\t, line width=1.3, color = cb-green]
    ({t}, {sqrt(6 - t^2)/sqrt(12)});
    \addplot+[no markers, samples=200, samples y=0, domain=-sqrt(6):sqrt(6), variable=\t, line width=1.3, color = cb-green]
    ({t}, {-(sqrt(6-t^2)/sqrt(12)});
    \draw[cb-green, line width = 1.3] (0,-1.48) -- (0, 1.44);

    \addplot+[no markers, samples=200, samples y=0, domain=-sqrt(1/2):sqrt(1/2), variable=\t, line width=1.3, color = cb-purple]
    ({t}, {1/2*sqrt(3/2)*sqrt(1 - 2*t^2)});
    \addplot+[no markers, samples=200, samples y=0, domain=-sqrt(1/2):sqrt(1/2), variable=\t, line width=1.3, color = cb-purple]
    ({t}, {-1/2*sqrt(3/2)*sqrt(1 - 2*t^2)});
    \draw[cb-purple, line width = 1.3] (-3.6,0) -- (3.4,0);
    
    \addplot [only marks, mark=x, line width = 1.3, mark size = 3pt] table {
    0  0
    2.44949  0
    -2.44949  0
    };
    \addplot [only marks, mark=square*, mark size = 2pt] table {
    0  0.612372
    0  -0.612372
    };

    \node at (2.3,-1.28) [rectangle, draw, fill = cb-palegrey, inner sep=1.5pt] {$\lambda = 4$};
    \end{axis}
\end{tikzpicture}

    
\fi
    \caption{The system \eqref{eq:system_crit_pts} for the function $V$ from Example \ref{ex:Ising_numerical_An}, for the values $\lambda = \frac{1}{2}, \frac{1}{4}, 4$, from left to right. The solutions are marked in black. The solutions that are (equally) closest to (but distinct from) the origin, are marked with squares. Notice the different scaling in the $y$-axis, for the sake of clarity.}
    \label{fig:sing_various_lambdas}
\end{figure}
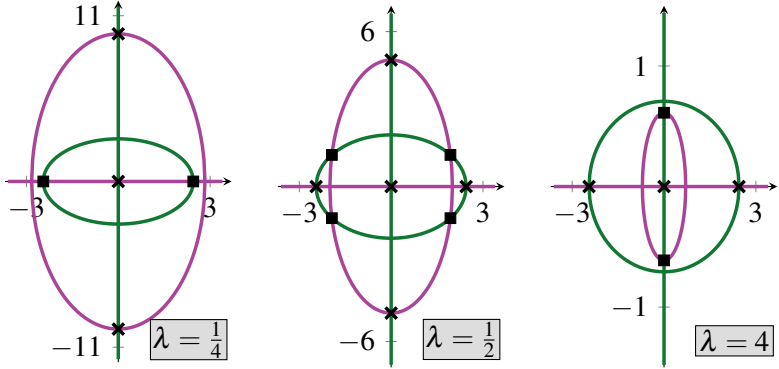
We exhibit the connection between the two collections $\Phi$ and $\Psi$ of critical points explicitly in our running example.
\begin{example}\label{ex:Ising_numerical_An}
    Let $V(x,y) = \frac{x^4}{4!} + \lambda \frac{x^2 y^2}{4} + \lambda^2 \frac{y^4}{4!}$ and $g(x,y) = -\frac{x^2}{2} - \frac{y^2}{2} + V(x,y)$.
    Consider the system of critical equations for $g$
    \begin{equation}\label{eq:system_crit_pts}
        {\color{cb-green} w = \frac{\partial V}{\partial w}(w,z)}, \quad
        {\color{cb-purple} z = \frac{\partial V}{\partial z}(w,z)},
    \end{equation}
    and its complex non-trivial solutions, for $\lambda>0$:
    \[
     (\pm \sqrt{6}, 0), \left(0,\pm \frac{\sqrt{6}}{\lambda}\right), \left( \pm \sqrt{\frac{9-3\lambda}{4\lambda}}, \pm \frac{\sqrt{9 \lambda -3}}{2\lambda} \right).
    \]
    Among these solutions, some are real for every $\lambda>0$. 
    The last type of singular points is real if and only if $\lambda \in [\frac{1}{3}, 3]$. We get
    \[
    \Psi = \begin{cases}
        (\pm\sqrt{6}, 0 ) & 0 < \lambda < \frac{1}{3}, \\
        \left(\pm \sqrt{\tfrac{9-3\lambda}{4\lambda}}, \pm \tfrac{\sqrt{9 \lambda -3}}{2\lambda} \right) & \frac{1}{3} < \lambda < 3, \\
        \left(0, \pm\tfrac{\sqrt{6}}{\lambda} \right) & \lambda > 3. 
    \end{cases}
    \]
    This is displayed in Figure \ref{fig:sing_various_lambdas}. The reader may check that rescaling each point in $\Psi$ to unit vector gives precisely two of the points in \eqref{eq:critV}.
\end{example}

\begin{remark}
    Lee--Yang theory studies the location of the roots of the polynomials $A_n$, when $n$ becomes large. This fascinating theory touches combinatorics, statistics and physics (see, e.g.,~\cite{MR2534100} for an overview). In the spirit of Lee--Yang theory, the two phase transitions $\lambda=\frac{1}{3},3$ in the running example can be detected also by looking at the asymptotic behavior of the roots of $A_n(\lambda)$ as $n\to \infty$.
    Using our algorithm from Proposition \ref{prop:algo}, we can compute the polynomials $A_n(\lambda)$ and find their roots numerically. This is the content of Figure \ref{fig:rootsAn}. The roots of these polynomials are all complex (except for $\lambda=-1$, for odd $n$) but they get closer and closer to the real values $\lambda=\frac{1}{3}$ and $\lambda=3$. 
    \begin{figure}[ht]
        \centering
\ifkeepslowthings
    \includegraphics[width=0.7\linewidth]{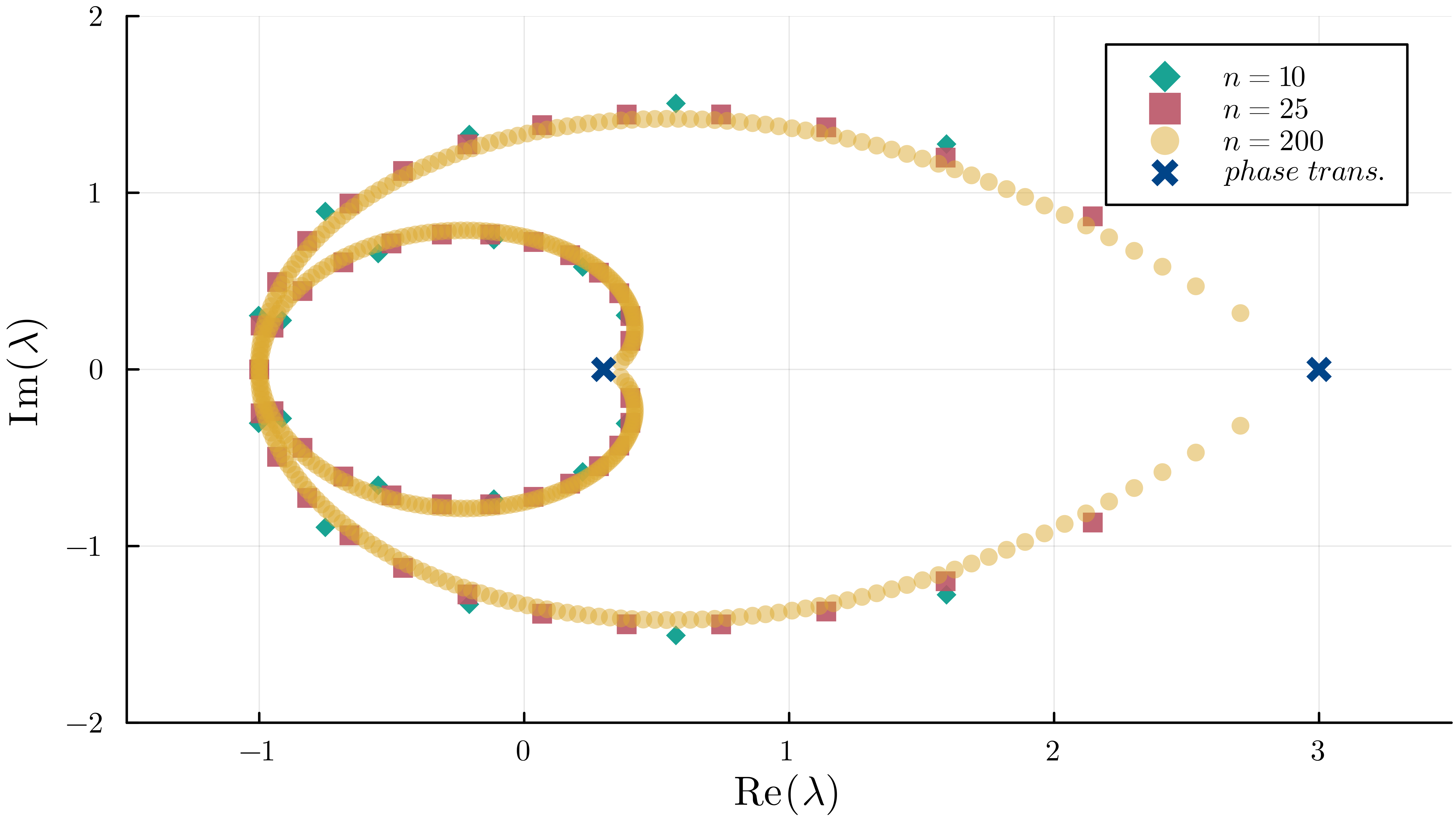}
\fi
        \caption{The roots of $A_n(\lambda)$ under the assumptions of Example \ref{ex:Ising_model_deg4_part1}, for $n=10, 25, 200$. The blue crosses are the phase transitions $\lambda=\frac{1}{3}, 3$.}
        \label{fig:rootsAn}
    \end{figure}
\end{remark}

Although Proposition \ref{prop:AnAsy} and Theorem \ref{thm:crit_g} assume $V$ to be homogeneous, the following example shows that this condition does not seem to be necessary. 
\begin{example}\label{ex:inhomogeneous}
    Take the inhomogeneous polynomial $V(x,y) = \frac{x^3}{3!} + \lambda \frac{x y^2}{2} + \lambda^2 \frac{y^4}{4!}$, with $\lambda>0$. 
    For $\lambda<\frac{1}{2}$, one can compute that $\Psi = \{(2,0)\}$; 
    for $\lambda>\frac{1}{2}$, one gets
    \[
        \Psi = \left\lbrace \left( \tfrac{4 \lambda-\sqrt{2\lambda(8 \lambda-3)}}{\lambda}, \pm \sqrt{6} \sqrt{\tfrac{1 - (4 \lambda-\sqrt{2\lambda(8 \lambda-3)}\,)}{\lambda^{2}}}\right) \right\rbrace.
    \]
    The formula for $A_n$ from Theorem \ref{thm:crit_g} would give 
    \begin{center}
    {\def\arraystretch{1.8}
    \begin{tabular}{c|c|c}
         & $\alpha(\lambda)$ & $c(\lambda)$ \\
        \hline
       $0<\lambda<\tfrac{1}{2}$  & $\frac{3}{2}$ & $\frac{1}{2\pi \sqrt{1-2\lambda}}$ \\
       $\lambda>\frac{1}{2}$  & $\frac{6 \lambda^2}{(8 \lambda-3) \left(16 \lambda-3 -4 \sqrt{2\lambda (8 \lambda-3)}\right)}$ & $\frac{1}{\pi} \sqrt{\frac{\lambda}{32 \lambda^2-12\lambda +2 \sqrt{2 \lambda (8 \lambda-3)(1 - 16 \lambda^2)}}}$ \\
    \end{tabular}
    }
\end{center}
This matches our numerical computations, see \cite{mathrepo}.
\end{example}
Based on the previous example and similar computations, we conjecture that Theorem \ref{thm:crit_g} is also valid for inhomogeneous $V(x,y)$, i.e., graphs that are not necessarily regular. 
\begin{conjecture}
    Let $g$ and $A_n$ be related as in the beginning of Section~\ref{sec:laplace}
    (i.e., $g(x,y)+\frac{x^2}{2} +\frac{y^2}{2}$ is not necessarily homogeneous). Let $\Psi$ be defined as before and assume that all points in $\Psi$ are non-degenerate. Then, 
    $$
    A_n \sim 
    \frac{1}{2\pi}
    \Gamma(n)
    \sum\limits_{(w,z) \in \Psi} \frac{(-g(w,z))^{-n}}{\sqrt{-\det H_g(w,z)}}
    \qquad \text{ as } \quad n \rightarrow \infty.
    $$
\end{conjecture}

\bibliography{references.bib}

\ifx\undefined\bysame
\newcommand{\bysame}{\leavevmode\hbox to3em{\hrulefill}\,}
\fi
\begin{thebibliography}{BMW24}

\bibitem[BB09]{MR2534100}
Julius Borcea and Petter Br\"and\'en, {\em The {L}ee-{Y}ang and {P}\'olya-{S}chur programs. {I}. {L}inear operators preserving stability}, Inventiones Mathematicae {\bf 177} (2009), no.~3, 541--569.

\bibitem[BH90]{berry1990hyperasymptotics}
Michael~V. Berry and Chris~J. Howls, {\em Hyperasymptotics}, Proceedings of the Royal Society of London. Series A: Mathematical and Physical Sciences {\bf 430} (1990), no.~1880, 653--668.

\bibitem[BIZ80]{bessis1980quantum}
Daniel Bessis, Claude Itzykson, and Jean-Bernard Zuber, {\em Quantum field theory techniques in graphical enumeration}, Advances in Applied Mathematics {\bf 1} (1980), no.~2, 109--157.

\bibitem[BMW24]{mathrepo}
Michael Borinsky, Chiara Meroni, and Maximilian Wiesmann, {\em \href{https://mathrepo.mis.mpg.de/BicoloredGraphs/index.html}{MATHREPO page: Bivariate Exponential Integrals and Edge-Bicolored Graphs}}, 2024, hosted by MPI MiS.

\bibitem[Bol04]{bollobas2004extremal}
B{\'e}la Bollob{\'a}s, {\em Extremal graph theory}, Courier Corporation, 2004.

\bibitem[Bor18]{Borinsky:2018mdl}
Michael Borinsky, {\em {Graphs in perturbation theory: Algebraic structure and asymptotics}}, Springer, 2018.

\bibitem[BV20]{borinsky2019euler}
Michael Borinsky and Karen Vogtmann, {\em The {E}uler characteristic of {${\rm Out}(F_n)$}}, Commentarii Mathematici Helvetici {\bf 95} (2020), no.~4, 703--748.

\bibitem[Cim12]{MR2989454}
David Cimasoni, {\em The critical {I}sing model via {K}ac--{W}ard matrices}, Communications in Mathematical Physics {\bf 316} (2012), no.~1, 99--126.

\bibitem[DC23]{MR4680248}
Hugo Duminil-Copin, {\em 100 years of the (critical) {I}sing model on the hypercubic lattice}, I{CM}---{I}nternational {C}ongress of {M}athematicians. {V}ol. 1. {P}rize lectures, EMS Press, Berlin, 2023, pp.~164--210.

\bibitem[LGZJ12]{le2012large}
Jean-Claude Le~Guillou and Jean Zinn-Justin, {\em Large-order behaviour of perturbation theory}, vol.~7, Elsevier, 2012.

\bibitem[Lin11]{lin2011algebraic}
Shaowei Lin, {\em Algebraic methods for evaluating integrals in bayesian statistics}, University of California, Berkeley, 2011.

\bibitem[Ski18]{skinner2018quantum}
David Skinner, {\em Quantum field theory {II}}, Lecture notes, Part III of the Mathematical Tripos, University of Cambridge (2018).

\bibitem[Wat09]{watanabe2009algebraic}
Sumio Watanabe, {\em Algebraic geometry and statistical learning theory}, vol.~25, Cambridge University Press, 2009.

\bibitem[Wor18]{MR3966531}
Nicholas Wormald, {\em Asymptotic enumeration of graphs with given degree sequence}, Proceedings of the {I}nternational {C}ongress of {M}athematicians---{R}io de {J}aneiro 2018. {V}ol. {IV}. {I}nvited lectures, World Sci.\ Publ., Hackensack, NJ, 2018, pp.~3245--3264.

\end{thebibliography}

\end{document}